\documentclass[10pt,twocolumn]{article} 
\usepackage{preprint}
\usepackage{times}
\usepackage{graphicx}
\usepackage{amssymb}
\usepackage{url}
\usepackage[table]{xcolor}
\usepackage{enumitem}  
\usepackage{natbib}
\usepackage{tabularx}
\usepackage{booktabs}
\usepackage[hidelinks]{hyperref}
\usepackage{amsmath}
\usepackage{tikz}
\usetikzlibrary{arrows.meta, positioning, calc}
\usepackage{amsthm}
\usepackage{fancyhdr}

\newcolumntype{Y}{>{\centering\arraybackslash}X}
\newtheorem{lemma}{Lemma}

\begin{document}

\title{A Fast and Effective Breakpoints Heuristic Algorithm for the Quadratic Knapsack Problem}

\author{D.~S.~Hochbaum$^1$, P.~Baumann$^2$, O.~Goldschmidt$^3$, Y.~Zhang$^1$\\
\\
$^1$IEOR Department, University of California, Berkeley, CA 94720, USA \\
$^2$Department of Business Administration, University of Bern, Engehaldenstr.\ 4, 3012 Bern, Switzerland\\
$^3$Riverside County Office of Education, Riverside, CA 92501, USA\\
\\
dhochbaum@berkeley.edu, philipp.baumann@unibe.ch, goldoliv@gmail.com, zhang@berkeley.edu \\
\\
This work has been published in the \href{https://www.sciencedirect.com/science/article/pii/S0377221724009615}{\color{blue}European Journal of Operational Research}
\\
\\
\url{https://doi.org/10.1016/j.ejor.2024.12.019}
\\
\\
The code of the algorithm is being developed on an on-going basis at
\\
\href{https://github.com/phil85/breakpoints-algorithm-for-qkp}{\color{blue}github.com/phil85/breakpoints-algorithm-for-qkp}.
}

\maketitle
\thispagestyle{empty}

\fancypagestyle{empty}{
	\fancyhf{} 
	\renewcommand{\headrulewidth}{0pt} 
	\renewcommand{\footrulewidth}{0pt} 
	\fancyfoot[C]{\copyright\ 2024 D.S.\ Hochbaum, P.\ Baumann, O.\ Goldschmidt and Y.\ Zhang.} 
}

\bibliographystyle{elsarticle-harv}

\begin{abstract}
The Quadratic Knapsack Problem (QKP) involves selecting a subset of elements that maximizes the sum of pairwise and singleton utilities without exceeding a given budget. The pairwise utilities are nonnegative, the singleton utilities may be positive, negative, or zero, and the node costs are nonnegative. We introduce a Breakpoints Algorithm for QKP, named QKBP, which is based on a technique proposed in \cite{Hoc09} for efficiently generating the concave envelope of the solutions to the relaxation of the problem for all values of the budget. Our approach utilizes the fact that breakpoints in the concave envelopes are optimal solutions for their respective budgets. For budgets between breakpoints, a fast greedy heuristic derives high-quality solutions from the optimal solutions of adjacent breakpoints. The QKBP algorithm is a heuristic which is highly scalable due to an efficient parametric cut procedure used to generate the concave envelope. This efficiency is further improved by a newly developed compact problem formulation. Our extensive computational study on both existing and new benchmark instances, with up to 10,000 elements, shows that while some leading algorithms perform well on a few instances, QKBP consistently delivers high-quality solutions regardless of instance size, density, or budget. Moreover, QKBP achieves these results in significantly faster running times than all leading algorithms. The source code of the QKBP algorithm, the benchmark instances, and the detailed results are publicly available on GitHub.
\end{abstract}

\section{Introduction}
The Quadratic Knapsack Problem (QKP) is to select from a given set of elements a subset that maximizes the sum of pairwise utilities and singleton utilities in the subset, so that the total cost of the subset does not exceed a given budget. The pairwise utilities are nonnegative, the singleton utilities may be positive, negative, or zero, and the node costs are nonnegative. The QKP problem is NP-hard since it generalizes the strongly NP-hard maximum clique problem. The Quadratic Knapsack Problem was introduced under this name by \cite{QKPgallo1980}. The problem arises in various application areas that include the maximum dispersion problem (\citealt{Kuby87, witzgall1988electronic, Kincaid92, Kuo93, Augca00,Aringhieri08,Hoc09,Aringhieri11}), the maximum diversity problem (\citealt{Ghosh96,Glover98,Silva04,Duarte07,Palubeckis07,Silva07,Wang12,Wang14,De2014,Zhou17,Marti21,parreno2021measuring,hocACDA2023breakpoints,spiers2023exact}), the collaborative team formation problem (\citealt{hocACDA2023breakpoints}) and the maximum benefit problem (\citealt{Hoc09}). Additional applications of QKP are related to the problem of maximum clique, such as the hidden clique and the densest subgraph of bounded size problems. QKP has also been applied in telecommunications (\citealt{witzgall1975}), selecting sites for satellite stations such that the global traffic between the stations is maximized, and in railways or freight handling terminals (\citealt{rhys1970selection}).

Although QKP has been studied extensively for more than four decades, there are no methods known to-date that can consistently solve the problem to optimality or very close to optimality for general large scale instances. We present here a new heuristic for the Quadratic Knapsack Problem that is based on the breakpoints algorithm, introduced in \cite{Hoc09} and recently used for the maximum diversity and maximum dispersion problems, \citealt{hocACDA2023breakpoints}. We demonstrate that this approach, referred to here as QKBP, for Quadratic Knapsack BreakPoints, provides optimal or near-optimal solutions on a wide range of benchmark instances within dramatically fast running times, providing several orders of magnitude speedups as compared to state-of-the-art approaches. Moreover, this performance is robust across different types of benchmark instances, varying sizes of benchmark instances, densities and budget levels. 

Our literature review (see Section~\ref{sec:literature}) indicates that existing algorithms for QKP have primarily been evaluated on standard QKP test-sets, and only a few algorithms have been compared to each other.  The research to-date contains no comprehensive computational comparison that systematically analyzes how the performance of the algorithms depends on parameters such as, graph density and budget values relative to the total weight of the nodes, across different types of test-sets.  This paper provides such a comprehensive study for the first time.

The QKP problem can be formalized as a graph problem.
A natural graph representation is as an undirected graph $G=(V,E)$ defined on a set of $n$ elements $V$, designated as nodes, and a set of $m$ pairs $E$, designated as edges. Each edge $[i,j]$ corresponds to a pair of nodes $i,j$ with a positive pairwise utility $u_{ij}$. There are singleton node utilities (unrestricted in sign) $u_{ii}$ and nonnegative node costs $q_i$ for every node $i\in V$, and a budget $B$. The problem is to find a subset $S \subset V$ so that the total cost of the nodes in $S$ does not exceed the budget $B$ and so that the sum of the weights of edges and singletons within $S$ is maximum.  

Another, equivalent graph representation of QKP, is as a {\em directed} graph, $G=(V,A)$  where for every pair $i,j$ with positive utility $u_{ij}$ we have an arc $(i,j)$ directed from $i$ to $j$ for $i<j$.
As will be shown here, in Section~\ref{sec:compact}, although equivalent, this directed version leads to a more efficient procedure.

We note that the choice of representing the problem on a directed graph versus an undirected graph is of practical importance. Our minimum cut network corresponding to the problem formulation on an undirected graph would have, for each pair $i,j$, two arcs, whereas for the directed graph formulation it would have only one arc. The directed graph representation reduces the running time of QKBP because the minimum cut procedure used, HPF (Hochbaum PseudoFlow, \citealt{Hoc08, Chandran09}), runs in practice in linear time relative to the number of arcs. This new, directed, formulation of QKP and the associated graph are described in Section \ref{sec:compact}. 

The standard formulation of the QKP is as a quadratic binary optimization problem. Let $x_i$ be a binary variable which is equal to $1$ if node $i$ is selected in $S$ and $0$ otherwise.  

{\renewcommand{\arraystretch}{1.5}
	\setlength{\arraycolsep}{5pt}
	\[\textup{(QKP)}\left\{\begin{array}{lll}
		\max & \displaystyle \sum_{(i,j)\in A} u_{ij}x_{i}x_{j} + \sum _{i\in V} u_{ii} x_i  \\
		\textup{s.t.} & \displaystyle \sum _{i\in V} q_i x_i  \leq B\\
		& x_i \in \{0, 1\} \qquad \forall \ i\in V \\
	\end{array}\right.\]}

\cite{QKPgallo1980} proposed to use the {\em Lagrangian relaxation} resulting from relaxing the budget constraint. Specifically, the Lagrangian relaxation, for a fixed value of the budget $B$ is

\[
L_B(\lambda) = \max _{x_i\in \{0,1\} ,\ i\in V} \Bigg\{ 
\sum_{(i,j)\in A} u_{ij} x_i x_j + \sum _{i\in V} u_{ii} x_i 
\]
\[
+ \lambda \Big(B- \sum _{ i\in V} q_i x_i \Big) \Bigg\}.
\]

Since for any $\lambda$  $L_B(\lambda) $ is an upper bound on the respective value of QKP, $L_B(\lambda ^* ) = \min _{\lambda \geq 0} L_B(\lambda)$ is the tightest upper bound generated from this relaxation. The procedure to identify $L_B(\lambda ^* )$ involves binary search to find the best value of the Lagrange multiplier,~$\lambda ^*$. This idea of using the Lagrangian relaxation to generate upper bounds has been used in most exact algorithms devised for the problem, most recently by \cite{spiers2023exact}

The breakpoints algorithm is related to the Lagrangian relaxation, yet it takes a different approach of generating the ``concave envelope" which is a piecewise linear function mapping the budget value to the corresponding optimal relaxation value. Rather than focusing on the value of the relaxation to be used as an upper bound, it utilizes the optimal solution sets at the breakpoints of the piecewise linear function adjacent to the value of the budget in order to generate close to optimal solutions for the respective QKP.	

The relaxed problem $L_B(\lambda)$ is polynomial time solvable, and moreover, the concave envelope describing the solutions for all values of $\lambda$ is generated in polynomial time. It was shown in \cite{Hoc09} that this entire collection of solutions, for all values of $\lambda$, is derived in the running time required to solve a minimum cut on an associated graph, using the parametric minimum cut procedure of \cite{Hoc08,WebHPF,WebHPFFULL,WebHPFSIM}. The statement that one can derive the solutions for all values of $\lambda$ appears surprising at first glance since the domain of $\lambda$ values is infinite. As shown in \cite{Hoc09}, there are at most $|V|$ different solution sets for all values of $\lambda$ and they are furthermore nested as the values of $\lambda$ go down, so all can be represented in $O(|V|)$ space and time. 

In our approach, the concave envelope is generated with an efficient parametric cut procedure, that has never been previously employed for QKP. Such a parametric cut procedure was theoretically described,  by \cite{gallo1989fast}, for the push-relabel algorithm, and later, by \cite{hochbaum1998pseudoflow,Hoc08}, for the HPF (Hochbaum PseudoFlow) algorithm.  Both methods have the same complexity, however there is no implementation of the parametric push-relabel algorithm. We use the parametric HPF implementation in our experiments. In general there are two variants of the parametric cut procedure. The {\em fully parametric} variant generates all the breakpoints (see \citealt{WebHPFFULL});
The {\em simple parametric} variant takes as input a sequence of values of $\lambda$, or a sequence of source adjacent capacities and sink adjacent capacities that are monotone non-increasing on one side, and monotone non-decreasing on the other, \cite{WebHPFSIM} and outputs the minimum cut (and maximum flow) solution for the values of $\lambda$ where the solution changes - the breakpoints. We implement here a version of simple parametric HPF that has a streamlined interface for reading the input and that outputs the concave envelope of QKP.

The contributions made in this paper include:
\begin{itemize}[nosep]
	\item Providing an algorithm for QKP which is the first to deliver solutions to very large scale benchmark instances within a few seconds and often is orders of magnitude faster than known leading algorithms.
	\item Demonstrating a principled approach for generating a compact formulation of QKP, using the theory of monotone integer programs.
	\item Providing a streamlined implementation of the simple parametric cut procedure for the Quadratic Knapsack Problem. 
	\item Generating new benchmark instances for QKP derived from team formation, large size synthetic data sets, facility dispersion and generalized QKP benchmark instances. In all these, we systematically vary the budget fraction and the graph density and show how variations in the budget values and densities affect the performance of leading algorithms. 
	\item Conducting an extensive experimental study on a broader collection of diverse data sets than have ever been previously studied and comparing our approach (QKBP) to seven known competitive algorithms.
	\item We make all data and algorithms associated with our experimental analysis publicly available on Github. This includes the \href{https://github.com/phil85/benchmark-instances-for-qkp}{\color{blue} benchmark instances} generated, the \href{https://github.com/phil85/breakpoints-algorithm-for-qkp}{\color{blue} QKBP algorithm implementation} including the interface with the parametric minimum cut procedure, and the \href{https://github.com/phil85/results-for-qkp-benchmark-instances}{\color{blue} detailed results} for all instances. 
\end{itemize}

The paper is organized as follows. In the next section, we provide a literature review, followed by preliminaries and notation in Section~\ref{sec:prelim}. Section~\ref{sec:Lagrange-envelope} describes the Lagrangian relaxation of QKP and how it is related to the concave envelope. The QKP formulation and the compact formulation as well as the related graph constructions are given in Section~\ref{sec:compact}. The parametric network and the parametric cut procedure for the $\lambda$ relaxation is given in Section \ref{sec:parametric} and the breakpoints algorithm QKBP is described in Section~\ref{sec:envelope}. In Section~\ref{sec:comparison}, we describe the experimental study, provide the details about the tested algorithms, the benchmark instances, and analyze the results. We conclude with several remarks in Section~\ref{sec:conclusions}.

\section{Literature review}\label{sec:literature}
There is a large body of literature on the Quadratic Knapsack Problem. Surveys by \cite{cacchiani2022knapsack}, \cite{pisinger2007quadratic} and \cite{kellerer2004knapsack} provide a comprehensive coverage of existing literature on QKP. 

We review here a selection of leading exact and heuristic techniques for the QKP with nonnegative utilities. The more general version of the problem with nonnegative utilities, the supermodular Knapsack Problem, is studied in \cite{gallo1989supermodular}. The version of the QKP that permits negative utilities is a different problem and also known as the knapsack problem with conflict pair constraints (see \citealt{pferschy2009knapsack, yamada2002heuristic, punnen2024knapsack}). 

\cite{caprara1999exact} proposed Quadknap (referred to here as the QK algorithm), an exact branch-and-bound algorithm using Lagrangian relaxation to compute the upper bound. The QK algorithm was tested, in \cite{caprara1999exact}, on a standardized benchmark set proposed by \cite{gallo1980quadratic} with up to 400 nodes and with graph densities of 25, 50, 75 and 100\%. Although the upper bounds obtained were typically within 1\% of the optimum, the algorithm performed best for high-density instances since the upper bounds are generally tighter for these cases. 
\cite{pisinger2007solution} proposed an algorithm that makes use of aggressive reduction techniques to reduce the size of instances. The algorithm builds on top of an improved version of the \cite{caprara1999exact} bound based on upper planes and a reformulation of QKP, and the \cite{billionnet2004mixed} bound based on Lagrangian decomposition. Although this algorithm was reported to perform very well, the code is no longer accessible and the results cannot be replicated today, as per private communication with the authors, \cite{Priv_Comm_Pisinger}, 


\cite{chen2017iterated} proposed an iterated “hyperplane exploration” approach (IHEA) that adopts the idea of searching over a set of hyperplanes defined by a cardinality constraint to delimit the search to promising areas of the solution space. On all small instances generated by \cite{billionnet2004mixed}, IHEA consistently achieved the known optimal solutions. When working with large instances of size 5,000 to 6,000 nodes, the algorithm was reported to outperform two other algorithms.

\cite{QKPfomeni2014dynamic} presented an algorithm that uses dynamic programming as a heuristic by modifying the classic dynamic programming  technique used for the linear knapsack problem. When tested against the benchmark set by \cite{gallo1980quadratic}, this heuristic, referred to here as the DP algorithm, is reported to achieve close to optimal results in most tested cases. \cite{fomeni2022cut} presented an improved Cut-And-Branch algorithm for QKP which consists of a cutting-plane phase followed by a branch and bound phase. The algorithm was shown capable of solving standard QKP instances with up to 800 elements within a five hour time limit. When tested against both the dispersion problem and the densest sub-graph instances of \cite{pisinger2007solution}, the algorithm performed fairly well with almost a third of these instances solved within the cutting plane phase, and a large majority of the remaining instances solved within a three hour time limit. For the much more difficult hidden clique instances, the algorithm is reported to find the optimal solution for instances with up to 200 elements within five hours. By combining dynamic programming with a local search procedure adapted and implemented in the space of lifted variables of the QKP, \cite{fomeni2023lifted} proposed a new deterministic heuristic algorithm for finding good QKP feasible solutions. This algorithm, called LDP, was reported to perform better than the dynamic programming algorithm suggested previously by \cite{QKPfomeni2014dynamic}, especially for dispersion problem and densest sub-graph problem instances.

The breakpoints algorithm concept was introduced in \cite{Hoc09} which provided the framework idea of the breakpoints concave envelope and the use of greedy for budgets that fall between breakpoints. However, \cite{Hoc09} used the simple formulation of QKP on a bipartite graph which leads to the relaxation formulation ($\lambda$-QKP1).  Our formulation here, ($\lambda$-QKP2) is much more compact: It leads to a graph on $O(n)$ nodes and $O(m)$ arcs versus $O(m+n)$  nodes and $ O(m)$ arcs. In addition, the results in \cite{Hoc09} are theoretical and address theoretical improvements in the complexity of the parametric flow, or cut, procedure. By comparison, we focus here on practical enhancements of the parametric HPF code, reducing drastically the storage and computation time required to prepare the input to the code, and for a much more compact graph than was addressed in \cite{Hoc09}.

\section{Preliminaries and notation}\label{sec:prelim}

Let $G=(V,A)$ be a simple directed graph without self-loops. Each arc $(i,j)$ corresponds to a pair of nodes $i,j$ with a pairwise utility $u_{ij}$ directed from $i$ to $j$ for $i<j$. Each node $i\in V$ has a non-negative cost $q_i$ and a utility $u_{ii}$ that could be positive or negative.  For two subsets of nodes, $D_1,D_2\subseteq V$,  we let $C(D_1,D_2)= \sum _{i \in D_1,j \in D_2, \ (i,j) \in A} u_{ij}$.  
With this notation $C(S,S) = \sum _{i,j \in S, (i,j) \in A} u_{ij}$ and $C(S,\bar{S}) = \sum _{i \in S,j \in \bar{S}(i,j) \in A} u_{ij}$.
Let $U(S)=\sum _{i\in S} u_{ii}$ and $q(S)=\sum _{i\in S} q_i $. 

QKP is then formulated as 
\begin{equation}
	\max _{\emptyset \subset S \subset V, q(S)\leq B} C(S,S) +U(S).
	\label{eqn:QKP}
\end{equation}

The relaxation $\lambda$-QKP problem is then to identify an optimal subset of nodes that maximizes the expression, omitting the constant term $\lambda B$:
\begin{equation}
	\max _{\emptyset \subset S\subset V} C(S,S) +U(S)- \lambda q(S).
	\label{eqn:lambda-QKP}
\end{equation}

Let $d_i^+$ denote the weighted out-degree of node $i$ in $G$:  $d_i^+= \sum _{j| (i,j) \in A} u_{ij}$ and $d^+(S)=\sum _{i\in S} d_i^+$. The following equation that is easy to verify will be used to generate a compact formulation of the $\lambda$-QKP problem, 
\begin{equation}
	d^+(S)=C(S,S)+ C(S,\bar{S}).
	\label{eqn:C(S,S)}
\end{equation}
In the sequel, we will refer to the number of elements $|V|$ as $n$ and the number of pairwise utilities, $|A|$, as $m$.

\section{The Lagrangian relaxation and the concave envelope of QKP}\label{sec:Lagrange-envelope}
Consider here the Lagrangian relaxation of the problem's budget constraint for a Lagrange multiplier $\lambda$:
$$\text{($\lambda$-QKP) }\max _{x_i\in \{0,1\} ,\ i\in V}  \sum_{(i,j) \in A} u_{ij}x_{i}x_j +\sum _{i\in V} (u_{ii} -\lambda q_i)x_i.$$
This relaxation omits from the objective function the term $B\lambda$ which is constant for every $\lambda$. Let $f(B)$ be the value of the optimal solution to QKP with budget $B$:\\
$f(B)=\max _{S\subset V } \{ C(S,S) +U(S)|q(S)\leq B \}$, and let $S_B=\arg\max _{S\subset V } \{ C(S,S) +U(S)|q(S)\leq B \}$ be the corresponding solution set.  We will refer to the optimal objective value as the ``benefit" of the budget $B$.  Figure \ref{fig:0concave} illustrates the values of $f(B)$ for all budgets.  Next consider a collection of all lines that lie above all these points.  The minimum piecewise linear line segments among this collection that  lie above all the optimal points, is known to be concave.  Any point where the line segment changes, and the slope becomes lower, is called a {\em breakpoint}, see Figure \ref{fig:concave}. We call the slope of the first line segment $\lambda _1$, the second $\lambda _2$, which is lower in value, etc., for a total of $p$ breakpoints.

\begin{figure*}
	\begin{minipage}[c]{0.5\linewidth}
		\includegraphics[width=\linewidth]{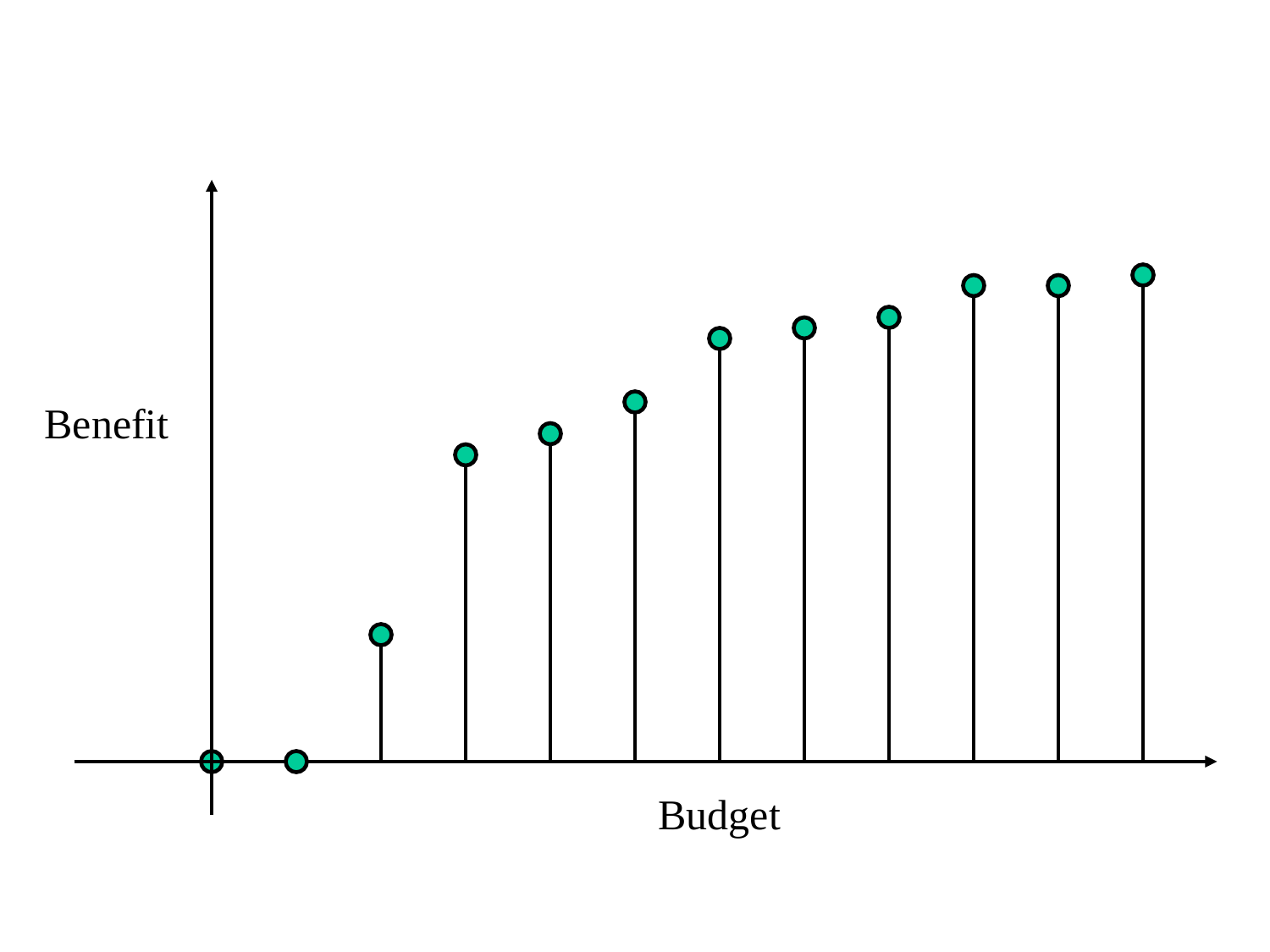}
		\caption{Optimal solutions for each budget}
		\label{fig:0concave}
	\end{minipage}
	\hfill
	\begin{minipage}[c]{0.5\linewidth}
		\includegraphics[width=\linewidth]{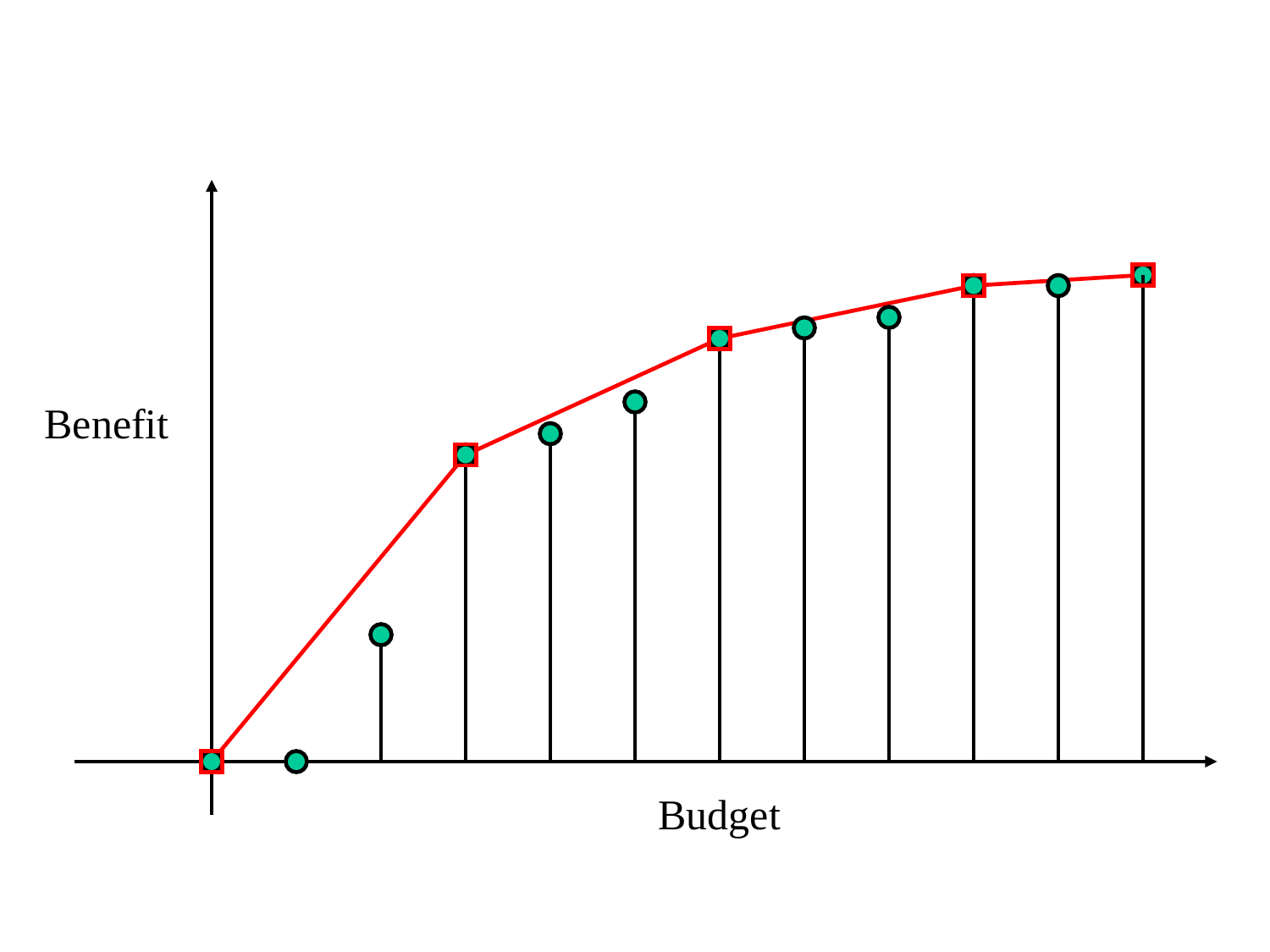}
		\caption{The concave envelope and the breakpoints, indicated as squares.}
		\label{fig:concave}
	\end{minipage}%
\end{figure*}

\begin{figure}
	\centering
	\includegraphics[scale=0.3]{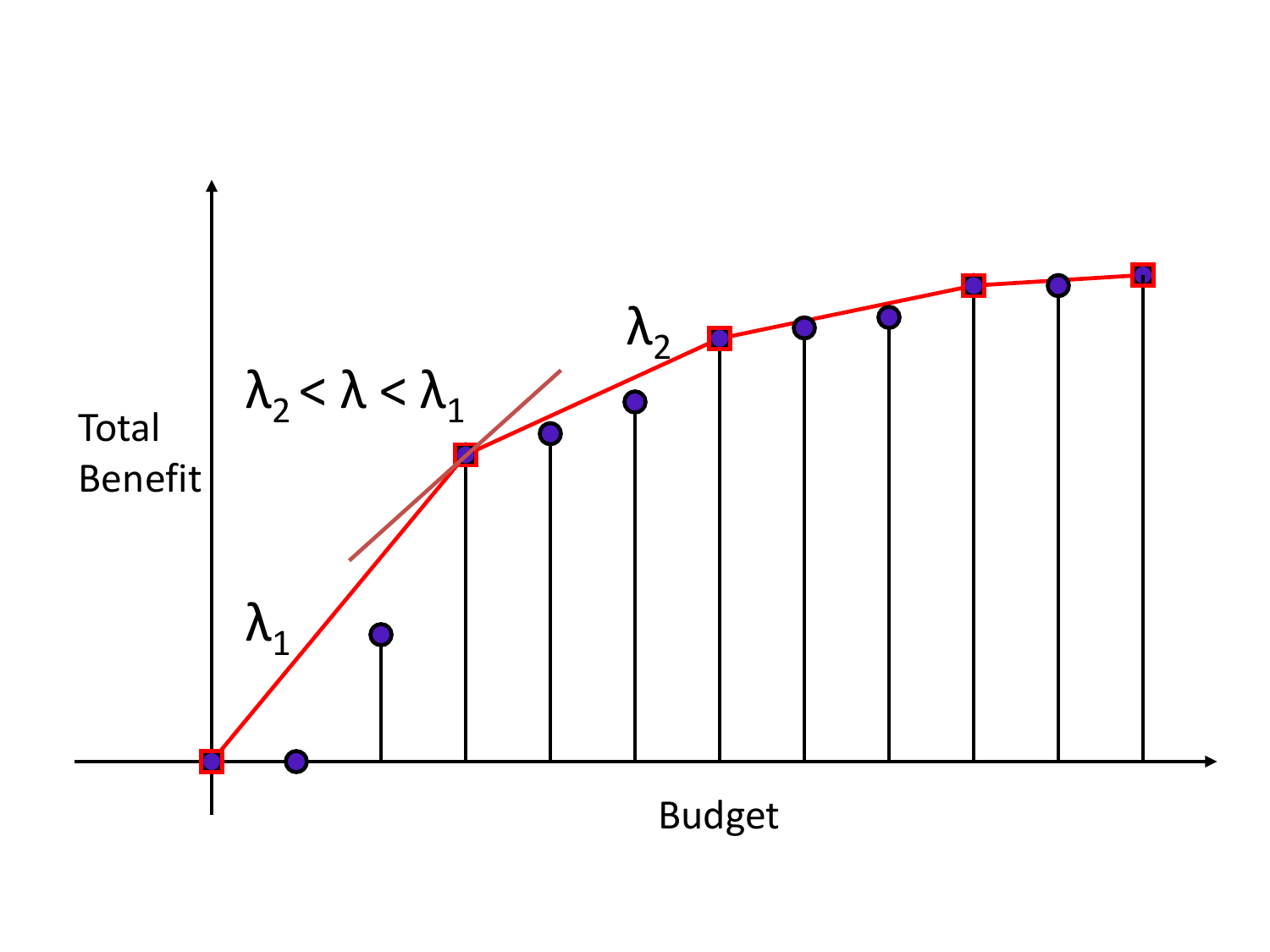}
	\vspace{-0.3in}
	\caption{The values of $\lambda$ associated with the concave envelope.}
	\label{fig:envelope}
\end{figure}

The properties of the concave envelope were studied, in the context of the dynamic evolution problem in \cite{Hoc09}. That paper addressed how the maximum benefit solution evolves as the budget increases.  These properties include:

\begin{itemize}[nosep]
	\item At the breakpoints of the envelope the solutions are optimal. 
	\item The first breakpoint -- the smallest positive budget breakpoint -- corresponds to the densest subgraph, which attains the largest ratio of the benefit to the budget.
	\item  The breakpoints correspond to solutions that are {\em nested} -- the solution set at one breakpoint is a subset of the solution sets for larger budgets' breakpoints.
	\item The number of breakpoints is at most $n$, the number of elements, or nodes, in the graph $G$ or since the origin is considered a breakpoint it could go up to $n+1$. 
	\item The envelope describes an upper bound on the optimal value at any level of the budget.
	\item If there are optimal solutions that lie on the line segments of the envelope, a method that uses the output of HPF described in \cite{Hoc09}, generates such solutions, in constant time per solution. 
\end{itemize}

\section{The formulation, compact formulation and solving $\lambda$-QKP as a minimum cut}\label{sec:compact}
The $\lambda$-QKP belongs to the class of monotone integer programming problems (IPM). IPM are (linear) integer programming problems on at most 3 variables per constraint, where two of the variables, the $x$-variables, appear with opposite sign coefficients, and a third variable, a $z$-variable, if included, can appear in at most one constraint. The coefficient of the third variable, in the objective function, must be non-negative for minimization problems, or non-positive for maximization problems.  
Integer problems in this class are solvable in polynomial time as a minimum cut on an associated graph, \cite{Hoc02}.  This is done by first reducing the IPM to an equivalent IPM on binary variables called the {\em s-excess} problem, which has a constraint matrix that is totally unimodular.  Any s-excess problem has an associated $s,t$-graph where the respective minimum $s,t$-cut solution provides the optimal solution. 

Consider the (linear) integer programming formulation of QKP to find an optimal solution set $S$. Define the binary variables: $x_i=1$ if item $i\in S$ and $0$ otherwise. Let $y_{ij}=1$ if both $i$ and $j$ are in $S$, and $0$ otherwise. With this notation the formulation of $\lambda$-QKP is, 

{\renewcommand{\arraystretch}{1.5}
	\setlength{\arraycolsep}{5pt}
	\[\textup{($\lambda$-QKP1)}\left\{\begin{array}{lll}
		\max & \displaystyle \sum _{(i,j) \in A} u_{ij}y_{ij} +\sum_{i\in V} (u_{ii} -\lambda q_i) x_i  \\
		\textup{s.t.} &  y_{ij} \leq x_i  \qquad \forall \ (i,j) \in A \\
		& y_{ij} \leq x_j  \qquad \forall \ (i,j) \in A\\
		& x_i \in \{0, 1\} \qquad \forall \ i\in V \\
		& y_{ij} \in \{0, 1\} \qquad \forall (i,j) \in A
	\end{array}\right.\]}

This formulation is a monotone integer program on at most two variables per inequality.   The associated constructed graph 
is a graph where each variable is represented by a node resulting in two types of nodes.  On one side of the bipartition there are $|A|$ nodes representing the $y_{ij}$ variables, and on the other side $|V|$ nodes representing the $x_j$ variables. Every $y_{ij}$ node has an incoming arc from the source $s$ of capacity $u_{ij}$ and 
two arcs of infinite capacity going to the nodes corresponding to $x_i$ and $x_j$.  Every node $x_j$  has two arcs, one from $s$ of capacity $\max \{ u_{ii} -\lambda q_i ,0\}$ and one going to the sink $t$ of capacity $-\min \{ u_{ii} -\lambda q_i ,0\}= \max \{ \lambda q_i -u_{ii},0\}$.  Note that this graph is a bipartite network only when $u_{ii}\leq 0$, in  which case it represents the selection problem (see \citealt{balinski1970selection,rhys1970selection}).
The formulation is that of a {\em monotone integer program in two variables per inequality} and as such it is solved as a minimum cut on the associated graph shown in Figure \ref{fig:graph2} where the variable associated with every node  that falls in the source set of the minimum cut taking the value $1$, and zero otherwise.

\begin{figure}
	\centering
	\resizebox{8cm}{!}{%
		\begin{tikzpicture}[
    node distance=3cm,
    mynode/.style={circle, draw, minimum size=0.8cm},
    myarrow/.style={-Latex},
]

\node[mynode] (s) {$s$};
\node[mynode, above right=1.5cm and 2cm of s] (yij) {$y_{ij}$};
\node[mynode, below right=1.5cm and 2cm of s] (ypq) {$y_{pq}$};
\node[mynode, right=8cm of s] (t) {$t$};
\node[mynode, above right=0.5cm and 2cm of yij] (xi) {$x_{i}$};
\node[mynode, below right=0.5cm and 2cm of yij] (xj) {$x_{j}$};
\node[mynode, above right=0.5cm and 2cm of ypq] (xp) {$x_{p}$};
\node[mynode, below right=0.5cm and 2cm of ypq] (xq) {$x_{q}$};
\node[above=0.8cm of yij, font=\tiny] (label) {``edge" nodes};
\node[above=0.1cm of xi, font=\tiny] (label) {``node" nodes};

\draw[myarrow] (s) -- (yij) node[midway,  above,font=\scriptsize] {$u_{ij}$};
\draw[myarrow] (s) -- (ypq) node[midway,  below,font=\scriptsize] {$u_{pq}$};
\draw[myarrow] (yij) -- (xi) node[midway, left,font=\scriptsize] {$\infty$};
\draw[myarrow] (yij) -- (xj) node[midway, left,font=\scriptsize] {$\infty$};
\draw[myarrow] (ypq) -- (xp) node[midway, left,font=\scriptsize] {$\infty$};
\draw[myarrow] (ypq) -- (xq) node[midway, left,font=\scriptsize] {$\infty$};
\draw[myarrow] (xj) -- (t) node[midway, left,font=\scriptsize] {};
\draw[myarrow] (xp) -- (t) node[midway, left,font=\scriptsize] {};
\draw[myarrow] (xq) -- (t) node[midway, left,font=\scriptsize] {};
\draw[myarrow] (s) to node[midway, left,font=\scriptsize] {} (xj);
\draw[myarrow] (s) to node[midway, left,font=\scriptsize] {} (xp);


\draw[myarrow] (xi) -- (t) node[midway,  right,font=\scriptsize] {$max\{-w_i,0\}$};
\draw[myarrow] (s) to[bend left=35] node[pos=0.3,  above left,font=\scriptsize] {$max\{w_i,0\}$} (xi);
\draw[myarrow] (s) to[bend right=35] node[pos=0.8,  above left,font=\scriptsize] {} (xq);

\foreach \n in {1,2,3} {
    \node at ($(xi)!0.3+\n*0.1!(xj)$) {$\cdot$};
}

\foreach \n in {1,2,3} {
    \node at ($(xj)!0.3+\n*0.1!(xp)$) {$\cdot$};
}

\foreach \n in {1,2,3} {
    \node at ($(xp)!0.3+\n*0.1!(xq)$) {$\cdot$};
}

\foreach \n in {1,2,3,4,5,6,7,8} {
    \node at ($(yij)!0.05+\n*0.1!(ypq)$) {$\cdot$};
}
\node[mynode, below = 4cm of s] (s2) {$s$};
\node[mynode, below=4cm of t] (t2) {$t$};
\node[mynode](xi2) at ($(s2)!0.5!(t2)$) {$x_i$};

\draw[myarrow] (s2) -- (xi2) node[midway,  above,font=\scriptsize] {$max\{u_{ii}-\lambda q_i,0\}$};
\draw[myarrow] (xi2) -- (t2) node[midway,  above,font=\scriptsize] {$max\{\lambda q_i-u_{ii},0\}$};
\end{tikzpicture}
	}
	\caption{The two variables per inequality formulation graph for $\lambda$-QKP1}
	\label{fig:graph2}
\end{figure}
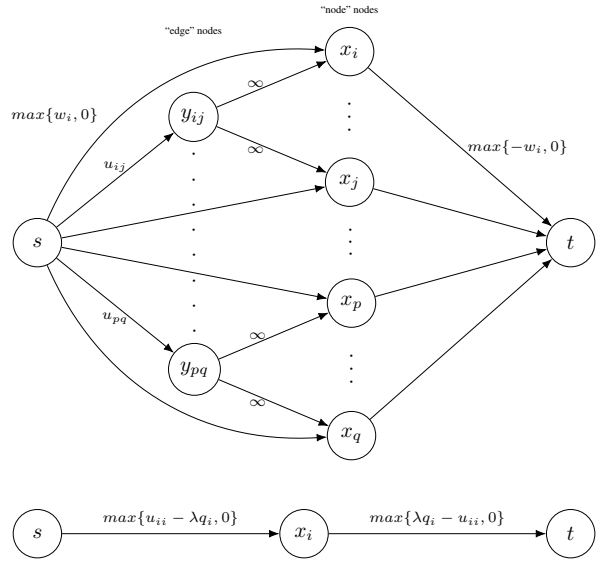

To introduce a compact formulation for solving the $\lambda$-QKP problem we use the formulation of the problem as,
$$\max _{\emptyset \subset S\subset V}C(S,S) +\sum _{i \in S} u_{ii}- \lambda\sum _{i \in S} q_i. $$ 
Recall that $d^+_i$ is the weighted out-degree of node $i$ in $G$:  $d_i^+= \sum _{j| (i,j) \in A} u_{ij}$.

\begin{lemma}
	The $\lambda$-QKP problem is equivalent to 
	\begin{equation}
		\max _{\emptyset \subset S\subset V} \sum _{i \in S}(d_i^++u_{ii}-\lambda q_i) - C(S,\bar{S}) . 
		\label{eqn:QKP-s-excess}
	\end{equation}
	
\end{lemma}

\begin{proof}
	From equation (\ref{eqn:C(S,S)}), for any subset of nodes $S \subset V$,
	$$d^+(S)=C(S,S)+ C(S,\bar{S}).$$
	Therefore,
	$$C(S,S) +U(S) - \sum _{i \in S} \lambda q_i 
	=d^+(S) -C(S,\bar{S}) + \sum _{i \in S}u_{ii} - \sum _{i \in S}\lambda q_i.$$  
	Rearranging the terms, this expression is equivalent to (\ref{eqn:QKP-s-excess}).
\end{proof} 

The problem (\ref{eqn:QKP-s-excess}) is an instance of the {\em s-excess} problem, 
which is a monotone integer program solved as a minimum cut on a graph of the size of $G$, \cite{Hoc08}:
The s-excess problem is defined on a directed graph $G=(V,A)$ with non-negative arc capacities and with node weights $w_i$ that can be positive or negative:
\vspace{-0.1cm}
\begin{equation*}
	(\mbox{s-excess} ) \hspace{0.6cm}  \max _{S\subseteq V} \sum _{i \in S}w_i -C(S,\bar{S}) .
\end{equation*}
For problem (\ref{eqn:QKP-s-excess}), $$w_i= d^+_i+u_{ii}-\lambda q_i.$$  
The formulation of the problem utilizes, in addition to variables $x_i$, the binary variables $z_{ij}$ that are equal to $1$ if $x_i=1$ and $x_j=0$ for each arc in the graph $(i,j)\in A$.

{\renewcommand{\arraystretch}{1.5}
	\setlength{\arraycolsep}{5pt}
	\[\textup{($\lambda$-QKP2)}\left\{\begin{array}{lll}
		\max & \displaystyle \sum_{i\in V} (d_i^++u_{ii}-\lambda q_i)x_i - \sum _{(i,j)\in A} u_{ij}z_{ij}\\
		\textup{s.t.} 		 & x_i-x_j \leq z_{ij}  \quad \forall\ (i,j)\in A\\
		& x_i \in \{0, 1\} \quad \forall\ i\in V \\
		& z_{ij} \in \{0, 1\} \quad \forall\ 	(i,j)\in A
	\end{array}\right.\]}

The associated $s,t$-graph $G_{st}=(V\cup \{s,t\},A\cup A_s \cup A_t)$ for this s-excess problem is constructed as follows:  Add to the graph $G$ a source node $s$ and a sink node $t$; each arc in the graph $(i,j)\in A$ has the capacity $u_{ij}$;
Add a set of arcs $A_s$ that go from the source node $s$ to nodes $i$ with $w_i>0$ and have capacity $u_{si}=w_i$. The sink adjacent arcs $A_t$ go from nodes $j$ with $w_j<0$ to the sink $t$, with capacity $u_{jt}=|w_j|=-w_j$.  This graph is shown in Figure \ref{fig:graph3} below. 

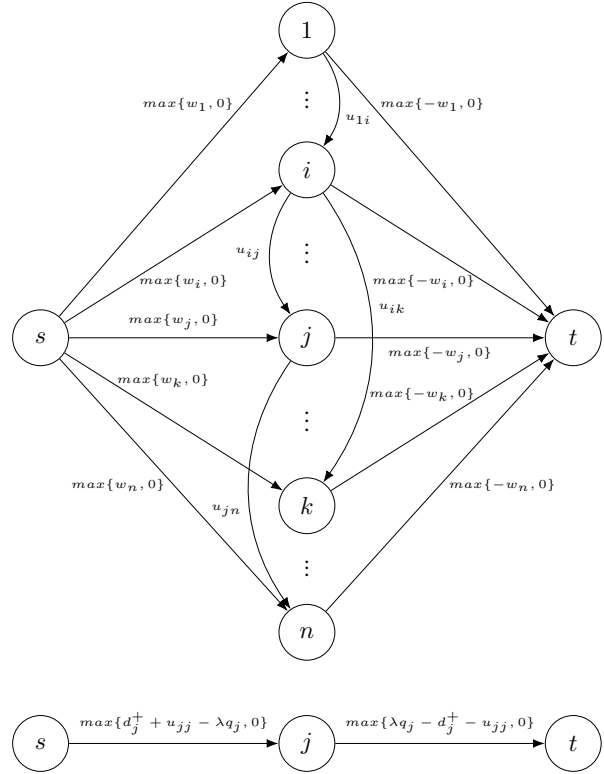
\begin{figure}
	\centering
	\resizebox{8cm}{!}{%
\begin{tikzpicture}[
    node distance=3cm,
    mynode/.style={circle, draw, minimum size=0.8cm},
    myarrow/.style={-Latex}
]

\node[mynode] (k) {$k$};
\node[mynode, above=4cm of k] (i) {$i$};
\node[mynode, above=6cm of k] (1) {$1$};
\node[mynode, below=1cm of k] (n) {$n$};

\node[mynode] (j) at ($(i)!0.5!(k)$) {$j$};
\node[mynode, left=of j] (s) {$s$};s
\node[mynode, right=of j] (t) {$t$};

\draw[myarrow] (s) -- (1) node[pos = 0.8,  left,font=\tiny] {$max\{w_1,0\}$};
\draw[myarrow] (1) -- (t) node[pos = 0.2,  right,font=\tiny] {$max\{-w_1,0\}$};
\draw[myarrow] (s) -- (i) node[pos=0.3,  right,font=\tiny] {$max\{w_i,0\}$};
\draw[myarrow] (i) -- (t) node[pos = 0.7,  left,font=\tiny] {$max\{-w_i,0\}$};
\draw[myarrow] (s) -- (j) node[midway, above,font=\tiny] {$max\{w_j,0\}$};
\draw[myarrow] (j) -- (t) node[midway, below,font=\tiny] {$max\{-w_j,0\}$} ;

\draw[myarrow] (1) to[bend left=35] node[pos=0.7, right, font=\tiny] {$u_{1i}$} (i);

\draw[myarrow] (j) to[bend right=35] node[pos=0.6, left, font=\tiny] {$u_{jn}$} (n);

\draw[myarrow] (i) to[bend right=35] node[pos=0.5, left, font=\tiny] {$u_{ij}$} (j);

\draw[myarrow] (i) to[bend left=35] node[pos=0.4, right, font=\tiny] {$u_{ik}$} (k);

\draw[myarrow] (s) -- (k) node[pos=0.2,  right,font=\tiny] {$max\{w_k,0\}$};
\draw[myarrow] (k) -- (t) node[pos = 0.7,  left,font=\tiny] {$max\{-w_k,0\}$};

\draw[myarrow] (s) -- (n) node[midway,  left,font=\tiny] {$max\{w_n,0\}$};
\draw[myarrow] (n) -- (t) node[midway,  right,font=\tiny] {$max\{-w_n,0\}$};

\foreach \n in {1,2,3} {
    \node at ($(j)!0.4+\n*0.05!(i)$) {$\cdot$};
}

\foreach \n in {1,2,3} {
    \node at ($(k)!0.4+\n*0.05!(j)$) {$\cdot$};
}

\foreach \n in {1,2,3} {
    \node at ($(n)!0.4+\n*0.05!(k)$) {$\cdot$};
}
\foreach \n in {1,2,3} {
    \node at ($(i)!0.4+\n*0.05!(1)$) {$\cdot$};
}
\node[mynode, below = 5cm of s] (s2) {$s$};
\node[mynode, right=of s2] (j2) {$j$};
\node[mynode, right=of j2] (t2) {$t$};

\draw[myarrow] (s2) -- (j2) node[midway,  above,font=\tiny] {$max\{d_j^++u_{jj}-\lambda q_j,0\}$};
\draw[myarrow] (j2) -- (t2) node[midway,  above,font=\tiny] {$max\{\lambda q_j-d_j^+-u_{jj},0\}$};

\end{tikzpicture}
	}
	\caption{The ``s-excess" flow graph for $\lambda$-QKP2.}
	\label{fig:graph3}
\end{figure}

The following lemma is a special case of the general result for the s-excess problem: 
\begin{lemma}[\citealt{Hoc02}]
	$S^*$ is an optimal solution to $\lambda$-QKP2 defined on graph $G$ if and only if $S^*$ is the source set of a minimum cut in $G_{st}$.
\end{lemma}
\begin{proof}
	For a given $\lambda$  denote $w_i=d^+_i+u_{ii}-\lambda  q_i$. Let  $V^+ = \{i \in V | w_i > 0\}$, and $V^- = \{j \in V |w_j < 0\}$. Let ($s \cup S, t \cup
	T$) be any $s,t$ cut in $G_{st}$. Then the capacity of this cut is given by
	\begin{eqnarray}
		&&C\left(s \cup S, t \cup T\right) \nonumber\\
		&=& \sum_{\left(s,i\right) \in A_{s},  i \in T   }  u_{si}   +  \sum_{\left(j,t\right) \in A_{t},  j \in S   }   u_{jt}  + \sum_{i \in S,  j \in T} u_{ij} \nonumber\\
		&=& \sum_{ i \in  T \cap V^+  } w_i + \sum_{ j \, \in \, S \cap V^-  } -w_j + \sum_{i \in S, \,\, j \in T} u_{ij} \nonumber\\
		&=& \sum_{i \in V^+}  w_i - \sum_{i \in S \cap V^+}  w_i + \sum_{ j \in  S \cap V^-  } -w_j + \sum_{i \in S,  j \in T} u_{ij} \nonumber \\
		&=& W^+ - \sum_{j \in S} w_j + \sum_{i \in S, \,\, j \in T} u_{ij} \nonumber
	\end{eqnarray}
	\noindent
	Where $W^+$ is the sum of all positive weights in $G$, which is a constant. Therefore, minimizing $C\left(s
	\cup S, t \cup T\right)$ is equivalent to maximizing $\sum_{j \in S} w_j - \sum_{i \in S, \,\, j \in T} u_{ij} $, and
	we conclude that the source set of a minimum $s,t$ cut on $G_{st}$ is also a maximum $\lambda$-QKP2 set.
\end{proof}

\section{Solving $\lambda$-QKP for all values of $\lambda$: The parametric cut}\label{sec:parametric}
Solving $\lambda$-QKP for all values of $\lambda$ requires to solve the minimum cut problem in Figure~\ref{fig:graph3} for all values of $\lambda$.  This flow network is a {\em parametric} flow network in that the arcs adjacent to the source are monotone non-increasing in the value of $\lambda$ and the arcs adjacent to the sink are monotone non-decreasing in the value of $\lambda$, or vice versa.  For a flow network with this property, the maximum flows and minimum cuts for all values of $\lambda$  can be solved with a {\em parametric cut} (or parametric flow) procedure in the same complexity as a single minimum cut (or maximum flow). This is true for parametric functions that are linear, as is the case here, whereas for general monotone parametric functions there is an unavoidable additive factor of $n \log U$ where $U$ is the range for the values of $\lambda$.  There are only two such parametric cut procedures known.  One is based on the push-relabel method of \cite{gallo1989fast}, and the other is based on the HPF method (Hochbaum's PseudoFlow) presented in \cite{hochbaum1998pseudoflow,Hoc08,Chandran09,WebHPFSIM}.

Let the source set of the minimum cut solving $\lambda$-QKP be denoted by $S_{\lambda}$.  Then for $\lambda_1 > \lambda_2$ it is known that $S_{\lambda_1}\subseteq S_{\lambda_2}$ - the {\em nestedness property}.  There are at most $n$, the number of nodes in the graph, values of $\lambda$ where $S_{\lambda}$ changes.  Each value of $\lambda$ in which the solution changes, say $\lambda _j$, is called a breakpoint and the solution $S_{\lambda _j}$ is optimal for $\lambda _j$-QKP for budget $B_j=\sum _{i\in S_{\lambda _j}} q_i$.    Plotting the values of the total utility of $S_{\lambda_j}$ as a function of the cost of $S_{\lambda_j}$, $\sum _{i\in S_{\lambda_j }} q_i$ generates a concave piecewise linear monotone increasing function referred to as the {\em concave envelope} of the solutions, see Figure~\ref{fig:envelope}. 

\section{The breakpoints algorithm and the concave envelope}\label{sec:envelope}
We use the concave envelope to generate a feasible solution for a given budget $B$ as follows: If $B$ corresponds to the budget $B_\ell$ at the $\ell$th breakpoint in the concave envelope, then the source set associated with that breakpoint, $S_{\ell }$, is the optimal solution. 
If that is not the case and the budget $B$ does not correspond to a breakpoint budget, the algorithm identifies the two breakpoints adjacent to $B$. That is, let $B_{\ell} < B_{\ell +1}$, be the budgets of two consecutive breakpoints such that $B_{\ell} < B< B_{\ell +1}$. We then generate two feasible solutions, one by greedily adding nodes to $S_{\ell}$, and the other one by greedily removing nodes from $S_{\ell + 1}$. We refer to the greedy procedure that adds nodes to $S_{\ell}$ as the {\em greedy-left} algorithm and the one that removes nodes from $S_{\ell+1}$ as the {\em greedy right} algorithm. 

The {\em greedy-left} algorithm initializes the solution set $S^\ast=S_\ell$ and the corresponding total node weights $B^\ast=B_\ell$. 
The algorithm then determines the set of candidate nodes that can be added to $S^\ast$ without violating the budget, $S^C= \{ i \in V \setminus S^\ast| B^\ast + q_i \leq B\}$ . For each node $i \in S^C$ the algorithm calculates the relative change in utility 
$\delta_i=\frac{u_{ii} + \sum_{j \in S^\ast}u_{ij}}{q_i}$ that would result by adding node $i$ to $S^\ast$. The node with maximum relative change in utility, $i_{\max} = \arg \max_{i \in S^C}\delta_i$, is added to set $S^\ast$ and $B^\ast$ is increased by $q_{i_{\max}}$. The algorithm then updates set $S^C$ and, for each $i \in S^C$ adjusts the values of $\delta_i$  as follows : $\delta_i := \delta_i + \frac{u_{i_{\max} i}}{q_i}$.
The greedy-left algorithm repeats these steps to add additional nodes until the set of candidate nodes $S^C$ is empty. If the budget $B$ is smaller than the budget of the smallest positive budget breakpoint, $B<B_1$, then greedy-left algorithm initiates $S^\ast$ with the node that has the highest weighted degree.

The {\em greedy-right} algorithm initializes the solution set $S^\ast$ and the corresponding total node weights $B^\ast$ to be equal to $S_{\ell+1}$ and $B_{\ell+1}$, respectively. The algorithm then computes for each node $i \in S^\ast$ the relative, negative, change in utility $\delta_i=-\frac{\sum_{j \in S^\ast}u_{ij}}{q_i}$ that would result by removing node $i$ from $S^\ast$. The node with {minimum} absolute value $|\delta_i|$ 
relative change in utility, $i_{\max} = \arg \min_{i \in S^\ast}|\delta_i|$, is removed from set $S^\ast$ and $B^\ast$ is decreased by $q_{i_{\max}}$. Again, 
$\delta_i$s are updated by setting 
$\delta_i := \delta_i + \frac{u_{i_{\max} i}}{q_i}$. The greedy-right algorithm repeats these steps to remove additional nodes until $B^\ast \leq B$ after removing node $i_{\max}$. If $B^\ast < B$, the greedy-right algorithm calls the greedy-left algorithm to check if additional nodes can be added. 

We implemented the greedy-left and the greedy-right algorithms such that multiple budgets $B^k$ for $k=1, \ldots, K$ can be provided as an input. If more than one of these budgets, for example $B'$ and $B''$, lie between two consecutive breakpoints,  $B_{\ell} < B' < B'' < B_{\ell +1}$, the greedy-left and greedy-right algorithms avoid redundant computations and thus do not have to be reinitialized with $S_\ell$ or $S_{\ell + 1}$ for each of those budgets. 

\section{Experimental study}\label{sec:comparison}

We compare the performance of our breakpoints approach (QKBP) to leading approaches for the Quadratic Knapsack Problem. We measure the performance of each approach in terms of the value of the objective function and the running time. In Subsection~\ref{sec:implementation_bp_approach}, we describe the implementation of the breakpoints approach. In Subsection~\ref{sec:leading_approaches}, we list the leading algorithms. In Subsection~\ref{sec:problem_instances}, we describe the benchmark instances that we use. In Subsection~\ref{sec:numerical_results}, we report the numerical results. In Subsection~\ref{sec:results_discussion}, we state the main conclusions that can be drawn from the results. All experiments were executed on an HP workstation with two Intel Xeon CPUs with clock speed 3.30 GHz and 256 GB of RAM.

\subsection{Implementation of the breakpoints approach (QKBP)}\label{sec:implementation_bp_approach}

The breakpoints approach is implemented in Python 3.11 and uses a simple parametric cut procedure implemented in the programming language~C as a subroutine. The simple parametric cut procedure returns the source sets~$S_\lambda$ for a list of specific values of~$\lambda$. In contrast to a fully parametric cut procedure, the simple parametric cut procedure is not guaranteed to find all breakpoints. However, in our experiments we found that the available implementation of the simple parametric cut procedure is faster than the available implementation of the fully parametric cut, and finds most of the breakpoints. Figure~\ref{fig:implementation} visualizes the implementation as a flow chart. The algorithm takes as input the edge weights $u_{ij}$, the singleton utilities $u_{ii}$, the node weights $q_i$, and one or multiple budget values.
The simple parametric cut procedure \texttt{QKPsimparamHPF.exe} derives the values of the parameter $\lambda$ as follows: An obvious upper bound to the value of $\lambda$ is $\textup{ub}=\max_{i=1,\ldots,n}\frac{d^+_i+u_{ii}}{q_i}$, with $d^+_i$ denoting the weighted out-degree of node $i$. The software then computes, for a given integer parameter $p$, the $p$ equidistant values of $\lambda$ in the interval $[\textup{ub}, 0]$. Concerning the number of values of the parameter, $p$, we use $p=1{,}600$, which we found to deliver satisfactory results. However, this is an input parameter that can be set by the user.

\begin{figure*}
	\begin{center}
		\includegraphics[width=0.8\textwidth]{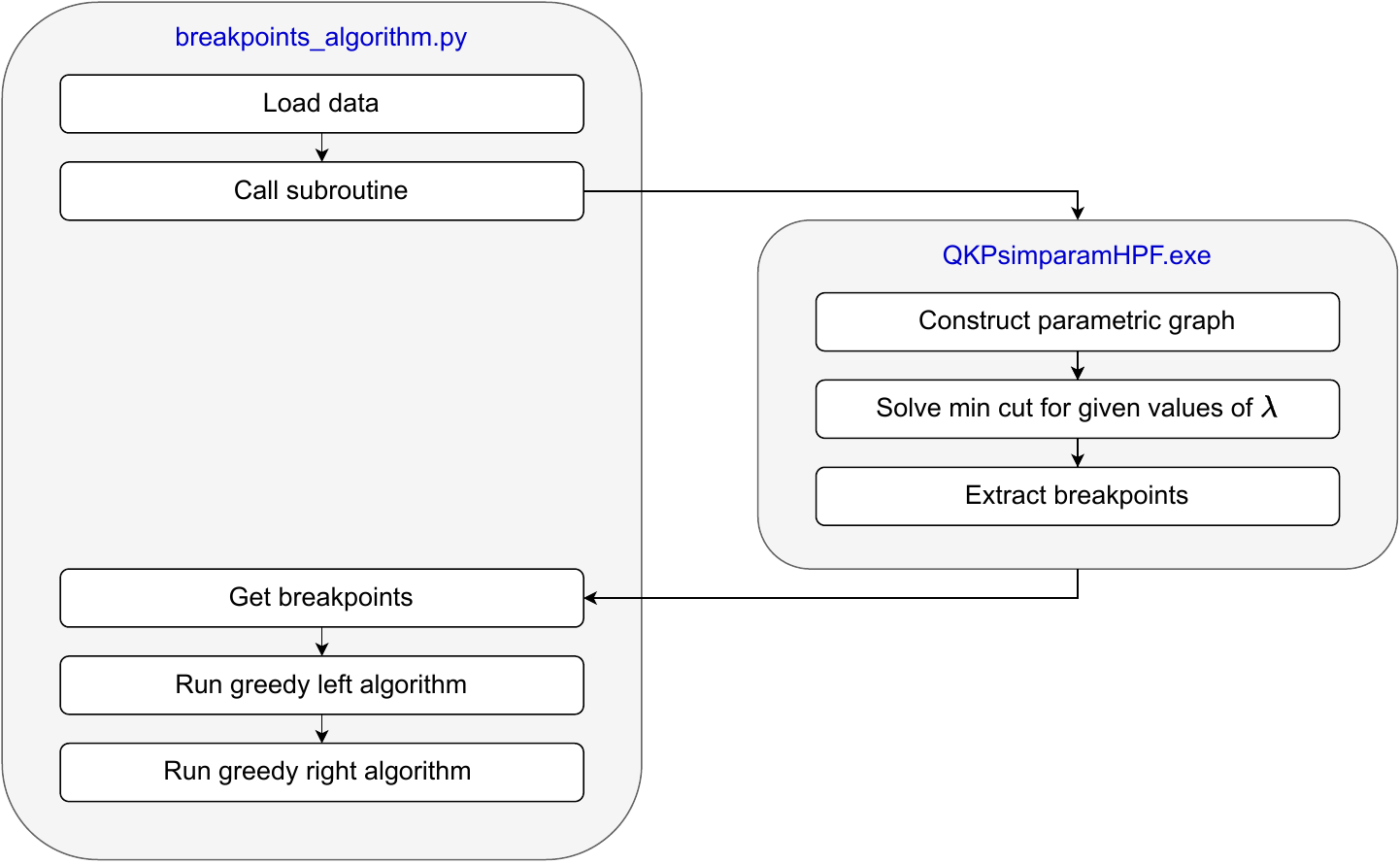}
	\end{center}
	\caption{Implementation of breakpoints approach}\label{fig:implementation}
\end{figure*}

\subsection{Leading approaches}\label{sec:leading_approaches}
We selected the leading approaches based on the recent review by \cite{cacchiani2022knapsack} of exact and heuristic methods for the Quadratic Knapsack Problem. According to this review the branch-and-bound algorithm Quadknap by \cite{caprara1999exact}, denoted here as QK, is a leading exact approach and the dynamic programming heuristic of \cite{QKPfomeni2014dynamic}, denoted here by DP, as well as the iterated hyperplane exploration approach of \cite{chen2017iterated}, denoted here by IHEA, are leading heuristics. We also include a metaheuristic that utilizes a multi-start greedy method that was reported to deliver high-quality solutions in \cite{julstrom2005greedy} for instances with up to 200 nodes. That method is referred to as the {\em relative-greedy} heuristic and is denoted here by RG. The most recent approach we include is the lifted-space dynamic programming algorithm, LDP, of \cite{fomeni2023lifted}. Finally, we include two approaches that are based on leading mathematical solvers, namely Gurobi and Hexaly. Table~\ref{table_tested_algorithms} provides for each approach the abbreviation that we use in the tables of this section, a reference to the original paper, the programming language of the implementation, and a link to the source code. Note that we re-implemented the RG heuristic of \cite{julstrom2005greedy} in Python since the original code is not publicly available. We applied the leading approaches with the default control parameter values or those recommended in the respective paper. 

\begin{table*}
	\small
	\begin{tabular}{p{5.25cm}lllp{1.6cm}}
		\toprule
		Name & Abbreviation & Reference & Programming language & Link to code \\ \midrule
		Breakpoints algorithm & QKBP & This paper & Python/C &  \href{https://github.com/phil85/breakpoints-algorithm-for-qkp}{\color{blue} GitHub} \\ \midrule
		Relative greedy heuristic & RG & \cite{julstrom2005greedy} & Re-implemented in Python & \href{https://github.com/phil85/greedy-algorithm-for-qkp}{\color{blue} GitHub} \\ \midrule
		Iterated hyperplane exploration approach & IHEA & \cite{chen2017iterated} & C++ & provided by authors \\ \midrule
		Lifted dynamic programming heuristic & LDP & \cite{fomeni2023lifted} & C & provided by author \\ \midrule		
		Dynamic programming-based heuristic & DP & \cite{QKPfomeni2014dynamic} & C & \href{https://sites.google.com/aims.ac.za/franklindjeumoufomeni/research/dynamic-programming-qkp-code}{\color{blue}Website} \\ \midrule
		Quadknap: specific branch-and-bound algorithm  & QK & \cite{caprara1999exact} & C & \href{http://hjemmesider.diku.dk/~pisinger/quadknap.c}{\color{blue} Website} \\ \midrule
		Gurobi-based approach & Gurobi & www.gurobi.com & C & \href{https://github.com/phil85/gurobi-based-approach-for-qkp}{\color{blue} GitHub} \\ \midrule
		Hexaly-based approach & Hexaly & www.hexaly.com & C++ & \href{https://github.com/phil85/hexaly-based-approach-for-qkp}{\color{blue} GitHub} \\ \midrule
	\end{tabular}
	\caption{Leading approaches included in the computational study.}\label{table_tested_algorithms}
\end{table*}

\subsection{Benchmarks}\label{sec:problem_instances}
To get a thorough understanding of the performance of the breakpoints algorithm (QKBP) compared to that of the leading approaches, we applied all approaches to a vast variety of more than $1{,}000$ benchmark instances. The benchmark instances originate in different application areas, consist of very large size instances that have not been tested to date in the literature on QKP, and vary systematically graph densities and budget values. The benchmark instances are organized in seven collections. Table~\ref{table_problem_instances} provides characteristics of each collection of benchmark instances. In the following, we briefly describe each collection. 

\begin{table*}
	\small
	\setlength{\tabcolsep}{8pt}
	\begin{tabularx}{\textwidth}{llrclll}
		\toprule
		Collection     & Nodes ($n$)  & \# Graphs & \# Budgets & Densities ($\Delta$)  & Introduced in                 & Link                                                                    \\ \midrule
		Standard-QKP   & 100--300     &    100 &       1 & 25--100               & \cite{gallo1980quadratic}    & \href{https://cedric.cnam.fr/~soutif/QKP/QKP.html}{\color{blue} Website}\\
		QKPGroupII   & 1000--2000     &    80 &       1 & 25--100               & \cite{yang2013effective}    & \href{https://leria-info.univ-angers.fr/~jinkao.hao/QKP.html}{\color{blue} Website}\\
		QKPGroupIII   & 5000--6000     &    40 &       1 & 25--100               & \cite{chen2017iterated}    & \href{https://leria-info.univ-angers.fr/~jinkao.hao/QKP.html}{\color{blue} Website}\\				
		Large-QKP        & 500--10,000  &     24 &       6 & 5--100                & This paper                    & \href{https://github.com/phil85/benchmark-instances-for-qkp}{\color{blue} GitHub}                                            \\
		Dispersion-QKP & 300--2,000   &     96 &       6 & 5--100                & \cite{pisinger2007solution}   & \href{https://github.com/phil85/benchmark-instances-for-qkp}{\color{blue} GitHub}                                            \\
		TeamFormation-QKP-1       & 1,021--9,269 &     14 &       6 & 0.06--2.15            & \cite{hocACDA2023breakpoints} & \href{https://github.com/phil85/benchmark-instances-for-qkp}{\color{blue} GitHub}                                            \\
		TeamFormation-QKP-2       & 1,000-10,000 &      6 &       6 & 12.38--13.43          & This paper                    & \href{https://github.com/phil85/benchmark-instances-for-qkp}{\color{blue} GitHub}                                            \\
		\bottomrule
	\end{tabularx}
	
	\caption{Characteristics of collections of benchmark instances. The range of densities ($\Delta$) is stated in percent.}\label{table_problem_instances}
\end{table*}

The \texttt{Standard-QKP} collection contains 100 randomly generated instances that are publicly available and often used as benchmark instances in the literature. The instances differ in the number of nodes in the graph ($n=100, 200, 300$) and the density of the graph ($\Delta=25\%, 50\%, 75\%, 100\%$). For all combinations of $n$ and $\Delta$ except (300, 75\%) and (300, 100\%), the collection includes ten instances, which were generated with the following procedure proposed by \cite{gallo1980quadratic}: for each node, an integer weight $q_i$ is chosen uniformly from $[1, 50]$, and for each pair of nodes $i$, $j$, where $i=1,\ldots,n$ and $j=i,\ldots,n$, an edge $[i, j]$ is added with probability $\Delta$. A random integer weight $u_{ij}$, uniformly chosen from $[1, 100]$, is assigned to each added edge. Note that this includes singleton utilities, when $i$=$j$. Finally, an integer budget $B$ is chosen uniformly in $[50, \sum_{j=1}^{n} q_j]$. This random problem generation procedure has been used in various papers such as \cite{caprara1999exact}, \cite{billionnet2004mixed}, \cite{pisinger2007solution}, \cite{yang2013effective} and \cite{chen2017iterated} to generate larger instances up to 6,000 nodes. The \texttt{QKPGroupII} and \texttt{QKPGroupIII} collections contain 80 and 40 large-sized instances, respectively. These instances have been generated with the above-described random procedure and are publicly available. 

The \texttt{Large-QKP} collection contains 144 instances with number of nodes varying from 500 to $10,000$ that we generated using a procedure very similar to the one proposed by \cite{gallo1980quadratic}. Unlike the \cite{gallo1980quadratic} procedure that has a single budget value for each graph, we generate multiple budget values for each graph where the budget $B$ is a fraction $\gamma$ of the sum of the node weights: $\lfloor \gamma \sum_{j=1}^n{q_j} \rfloor$, with $\gamma \in \{0.025, 0.05, 0.1, 0.25, 0.5, 0.75\}$. This allows us to analyze how the performance of the approaches is affected by the size of the budget relative to the sum of the node weights. For different combinations of numbers of nodes $n$ and graph densities $\Delta$, we generated 24 graphs, where $n$ ranges from 500--$10{,}000$ and the density $\Delta$ ranges from 5\% to 100\% (see Table~\ref{tbl:new_qkp}). For each graph, we created six instances corresponding to the six different budget values resulting in 144 instances. The resulting collection includes instances that are much larger in the number of nodes $n$ and have much lower densities compared to the instances that have previously been used in the QKP literature. \cite{schauer2016asymptotic} showed that under certain conditions (e.g., the edge weights are chosen independently), random generation procedures produce relatively easy instances for which a very simple heuristic, which sorts the nodes in non-decreasing order of their weights $q_i$ and includes the nodes greedily as long as they fit, will produce solutions whose objective function value is (asymptotically) very close to the optimal value. Since the generation procedure of \cite{gallo1980quadratic} satisfies the conditions established by \cite{schauer2016asymptotic}, we tested the approaches also on other collections which also contain instances where the edge weights $u_{ij}$ are not chosen independently and thus the conditions found in \cite{schauer2016asymptotic} are not satisfied.

The \texttt{Dispersion-QKP} collection contains 576 instances derived from instances of the dispersion problem (see \citealt{pisinger2007solution}) with number of nodes varying from $300$ to $2{,}000$. The dispersion problem is to determine the optimal placement of facilities among possible locations such that the pairwise distances between the facilities are maximized. This problem is a special case of a QKP where the distances between two locations $i$ and $j$ are the utilities $u_{ij}$, the cost of placing a facility at location $j$ is the $q_j$ and the available budget is $B$. \cite{pisinger2007solution} proposed a procedure to generate four types of dispersion benchmark instances that can be used as QKP instances. The four types differ with respect to the strategy that is used to determine the distances between the locations. Under the \texttt{geo} strategy, the locations are randomly selected within a 100 $\times$ 100 rectangle and the distances correspond to the Euclidean distances between these locations. Under the \texttt{wgeo} strategy, the locations are also randomly chosen within a 100 $\times$ 100 rectangle, and each location $j$ is assigned a random weight $\alpha_j$ chosen uniformly in the interval $[5, 10]$. The distances between two locations $i$ and $j$ are set to be $\alpha_i\alpha_j$ times the Euclidean distance between the locations $i$ and $j$. Under the \texttt{expo} strategy, each distance is randomly drawn from an exponential distribution with mean value 50 and under the \texttt{ran} strategy, each distance is an integer value chosen uniformly from the interval [1, 100]. Note that with the strategies \texttt{geo} and \texttt{wgeo}, the edge weights are not chosen independently. We follow the procedure described in \cite{pisinger2007solution} to generate QKP instances as follows: for a given number of nodes $n$, a given density $\Delta$, and a given distance computation strategy, we randomly choose an integer weight $q_j$ for each node from the interval [1, 100], and for each pair of nodes $i, j$, where $i=1,\ldots,n$ and $j=i+1,\ldots,n$, an edge $[i, j]$ is added with probability $\Delta$. Note that this procedure does not generate singleton utilities, meaning all $u_{ii}=0$. The utility of each added edge is computed according to the given distance computation strategy. For each of the four distance computation strategies, we generated 24 graphs by varying the number of nodes $n=300, 500, 1{,}000, 2{,}000$, and the densities of the graphs $\Delta \in \{5\%, 10\%, 25\%, 50\%, 75\%, 100\%\}$ resulting in a total of 96 graphs. For each graph, we generated six instances by setting the budget with the fractions $\gamma \in \{0.025, 0.05, 0.1, 0.25, 0.5, 0.75\}$ analogously to the instances in the collection \texttt{Large-QKP}.

The \texttt{TeamFormation-QKP-1} collection contains 84 instances derived from team formation benchmark instances of \cite{hocACDA2023breakpoints}. The team formation problem is to find a team of experts maximizing the collaboration utility of the team while satisfying extra constraints on required skills. The collaboration utility between experts $i$ and $j$ is computed by the Jaccard similarity $J(i, j)=\frac{|P_i \cap P_j|}{|P_i \cup P_j|}$, where $P_i$ denotes the set of projects that expert $i$ worked on. \cite{hocACDA2023breakpoints} transformed the four real team formation benchmark instances IMDB, DBLP, Bibsonomy, and StackOverflow into maximum diversity benchmark instances by omitting the extra constraints on required skills. In addition, they generated synthetic team formation benchmark instances for the maximum diversity problem. This was done by using lognormal distribution, with mean 4 and standard deviation 1, to select subsets of projects assigned to each expert. All of the synthetic instances have $70,000$ projects and $7,000$ experts. We add an integer cost $q_i$ to each expert $i$ selected randomly from the interval $[1, 10]$ for each of the four real and the ten synthetic team formation instances from \cite{hocACDA2023breakpoints}. We select the budget values $B$ with the fractions $\gamma \in \{0.025, 0.05, 0.1, 0.25, 0.5, 0.75\}$ analogously to the instances in the collection \texttt{Large-QKP}. These team formation QKP instances have a structure that is quite unique and are characterized by having a graph that is not fully connected and of very low density, often less than 1\% density. To the best of our knowledge, such instances have not been previously used to evaluate approaches for the QKP. 

The \texttt{TeamFormation-QKP-2} collection contains 36 randomly-generated instances of higher graph densities compared to the instances from \texttt{TeamFormation-QKP-1}. We use the procedure of \cite{hoc2023KDIR} for generating 6 new synthetic team formation benchmark instances with size $n \in \{1{,}000, 2{,}000, 4{,}000, 6{,}000, 8{,}000, 10{,}000\}$. We use the same lognormal distribution with mean equal to 4 and standard deviation equal to 1 to determine the projects assigned to each expert, but using $30,000$ rather than $70,000$ projects as in the synthetic instances of \texttt{TeamFormation-QKP-1}. This results in higher graph densities in the range $12.38\%-13.43\%$ as compared to the synthetic instances of \texttt{TeamFormation-QKP-1} that have graph densities in the range $0.06\%–2.15\%$. For each graph, we generated six instances by setting the budget with the fractions $\gamma \in \{0.025, 0.05, 0.1, 0.25, 0.5, 0.75\}$ analogously to the instances in the collection \texttt{Large-QKP}.

We made all instances and the codes to generate the synthetic benchmarks available on \href{https://github.com/phil85/benchmark-instances-for-qkp}{\color{blue} GitHub}.

\subsection{Experimental results}\label{sec:numerical_results}
In a first experiment, we applied all approaches to all instances with a time limit of 120 seconds. We will later increase this time limit to 3,600 seconds to analyze the impact on the solution quality. The proposed breakpoints approach never reached the time limit of 120 seconds, even for the largest instances, so we did not have to implement a stopping criterion. For the other approaches, we implemented the stopping criterion as follows. The RG approach is restarted $n$ times when applied to an instance with $n$ nodes, each time starting with a different node $i=1,\ldots,n$. We check the time limit before each restart and return the best solution found up to that point if the time limit has been reached. The DP and QK approaches do not provide the option to set a time limit. Therefore, we terminated the executable at the time limit and report that no solution was found within the time limit. For the IHEA approach, we checked the time limit after each iteration and returned the best solution found when the time limit was reached. For the Gurobi and Hexaly approaches, we set the solver time limit accordingly. For each instance, we report the detailed results on a \href{https://github.com/phil85/results-for-qkp-benchmark-instances}{\color{blue}GitHub} repository. In the remainder of this section, we present aggregated results for each of the seven collections of instances and draw conclusions based on them. 

We first present the results obtained for the collection \texttt{Standard-QKP} which contains 100 small instances with up to $n=300$ nodes. For each instance, we record the highest objective function value (OFV) achieved among all approaches and calculate the relative deviation of each approach's OFV from this best value. Table~\ref{tbl:standard_qkp} then reports for each approach and each combination of number of nodes $n$ and graph density $\Delta$ the average deviation from the best OFV in percent and the sum of the running times in seconds. There are 10 instances for each of those combinations. 

\begin{table*}
	\scriptsize
	\setlength{\tabcolsep}{5.1pt}
\begin{tabularx}{\textwidth}{rrr|crrrrrrr|crrrrrrr}
 \multicolumn{19}{c}{\texttt{Standard-QKP} collection} \\ \toprule
\multicolumn{3}{r}{} & \multicolumn{8}{c}{Avg deviation from best OFV (\%)} & \multicolumn{8}{c}{Sum runtime over ten budgets (s)} \\ \cmidrule(rl){4-11} \cmidrule(rl){12-19}\
$n$ & $\Delta$ & t$^{\textup{cut}}$ & \textbf{QKBP} & RG & IHEA & LDP & DP & QK & Gurobi & Hexaly & \textbf{QKBP} & RG & IHEA & LDP & DP & QK & Gurobi & Hexaly \\
\midrule
100 & 25 & 0.01 & 0.63 & 0.20 & \bfseries 0.00 & \textemdash & 0.26 & \bfseries 0.00 & \bfseries 0.00 & \bfseries 0.00 & \bfseries 0.02 & 1.2 & 13.1 & 219.8 & 3.2 & 3.3 & 4.7 & 33.0 \\
100 & 50 & 0.01 & 0.40 & 0.17 & \bfseries 0.00 & \bfseries 0.00 & 0.01 & \bfseries 0.00 & \bfseries 0.00 & \bfseries 0.00 & \bfseries 0.03 & 1.3 & 12.9 & 216.1 & 3.3 & 3.4 & 37.0 & 214.0 \\
100 & 75 & 0.01 & 0.76 & 0.38 & \bfseries 0.00 & \bfseries 0.00 & 0.01 & \bfseries 0.00 & \bfseries 0.00 & \bfseries 0.00 & \bfseries 0.03 & 1.3 & 12.5 & 218.6 & 3.2 & 3.0 & 146.0 & 675.0 \\
100 & 100 & 0.01 & 0.56 & 0.38 & \bfseries 0.00 & \bfseries 0.00 & \bfseries 0.00 & \bfseries 0.00 & \bfseries 0.00 & \bfseries 0.00 & \bfseries 0.03 & 1.4 & 11.9 & 212.3 & 3.1 & 1.9 & 258.6 & 586.0 \\
200 & 25 & 0.01 & 0.24 & 0.07 & \bfseries 0.00 & \textemdash & 0.02 & \textemdash & \bfseries 0.00 & \bfseries 0.00 & \bfseries 0.06 & 6.5 & 27.8 & 1,200.0 & 28.2 & 280.6 & 77.6 & 646.0 \\
200 & 50 & 0.01 & 0.17 & 0.08 & \bfseries 0.00 & \textemdash & 0.01 & \textemdash & \bfseries 0.00 & 0.05 & \bfseries 0.05 & 6.7 & 30.0 & 1,200.0 & 29.0 & 255.0 & 505.8 & 983.0 \\
200 & 75 & 0.01 & 0.56 & 0.25 & \bfseries 0.00 & \textemdash & 0.01 & \bfseries 0.00 & \bfseries 0.00 & 0.01 & \bfseries 0.05 & 5.2 & 35.9 & 1,124.2 & 27.0 & 12.9 & 570.1 & 1,200.0 \\
200 & 100 & 0.01 & 0.43 & 0.19 & \bfseries 0.00 & \textemdash & \bfseries 0.00 & \bfseries 0.00 & 0.01 & 0.43 & \bfseries 0.06 & 5.6 & 28.3 & 949.1 & 25.5 & 11.0 & 896.5 & 1,200.0 \\
300 & 25 & 0.02 & 0.40 & 0.16 & \bfseries 0.00 & \textemdash & 0.08 & \textemdash & \bfseries 0.00 & 0.01 & \bfseries 0.08 & 12.6 & 29.4 & 1,081.6 & 113.5 & 294.9 & 223.7 & 936.0 \\
300 & 50 & 0.02 & 0.18 & 0.03 & \bfseries 0.00 & \textemdash & \bfseries 0.00 & \textemdash & 0.16 & 0.89 & \bfseries 0.12 & 14.8 & 32.5 & 1,200.0 & 118.3 & 327.6 & 1,029.5 & 1,200.0 \\
 \midrule Avg &  &  & 0.43 & 0.19 & \bfseries 0.00 & \textemdash & 0.04 & \textemdash & 0.02 & 0.14 & \bfseries 0.05 & 5.7 & 23.4 & 762.2 & 35.4 & 119.4 & 375.0 & 767.3 \\
Min &  &  & 0.17 & 0.03 & \bfseries 0.00 & \bfseries 0.00 & \bfseries 0.00 & \bfseries 0.00 & \bfseries 0.00 & \bfseries 0.00 & \bfseries 0.02 & 1.2 & 11.9 & 212.3 & 3.1 & 1.9 & 4.7 & 33.0 \\
Max &  &  & 0.76 & 0.38 & \bfseries 0.00 & \textemdash & 0.26 & \textemdash & 0.16 & 0.89 & \bfseries 0.12 & 14.8 & 35.9 & 1,200.0 & 118.3 & 327.6 & 1,029.5 & 1,200.0 \\
\bottomrule \\ [-1.5ex] \multicolumn{8}{l}{\textbf{QKBP} is the contribution in this paper}
\end{tabularx}

	\caption{Results for the instances of the \texttt{Standard-QKP} collection: Each row represents ten graphs with the same values of $n$ and $\Delta$. The budget values are randomly selected in the interval $[50, \sum_{j=1}^{n} q_j]$. The abbreviation OFV stands for objective function value and the column $t^{\textup{cut}}$ reports the average running time of the simple parametric cut procedure in seconds. The time limit for each instance is 120 seconds. The hyphen (--) indicates that the respective approach did not find a solution for at least one of the ten instances within the time limit.}\label{tbl:standard_qkp}	
\end{table*}

The column $t^{\textup{cut}}$ reports the average running time of the simple parametric cut procedure used as a subroutine in the breakpoints approach, measured in seconds. Hence, the running time reported in column 12 does not include the time for the simple parametric cut procedure. We report the time for the simple parametric cut procedure separately in all tables in this section, because the breakpoints approach performs this procedure only once for multiple budget values for the same graph. This is a distinct advantage of our algorithm when it comes to sensitivity analysis of the solution as the value of the budget is modified. To analyze the trade-off between increasing the budget and benefiting from additional utility, our approach solves the simple parametric cut procedure only once instead of computing a solution for each budget value. The greedy right and left algorithms are not restarted when multiple budgets fall between the budgets of two consecutive breakpoints. This is in contrast to the other approaches, which require repeating the entire algorithm for each budget values. 

The hyphen (--) indicates that the respective approach did not find a solution for at least one of the ten instances with combination ($n$, $\Delta$) within the time limit. The last three rows of the table report the average, the minimum and the maximum value for each column. For each row, we highlight the best values (lowest deviation and lowest running time) in bold. For most instances, the exact approaches (QK, Gurobi, Hexaly) are able to identify an optimal solution within the time limit. Among the heuristic approaches, IHEA consistently produces the best solutions closely followed by the other approaches (QKBP, RG, DP). The LDP approach often reaches the time limit before it returns a solution. Our breakpoints approach (QKBP) is much faster than the other approaches and still delivers competitive solution quality. The RG approach is similar to the greedy-left algorithm that we use as a subroutine in our breakpoints approach. An important difference, however, is that the RG approach is applied multiple times, each time restarting with a different seed item, resulting in high running times for medium and large instances.

\begin{table*}
	\scriptsize
	\setlength{\tabcolsep}{4.65pt}
\begin{tabularx}{\textwidth}{rrr|rrrrrrrr|rrrrrrrr}
 \multicolumn{19}{c}{\texttt{QKP-GroupII} collection} \\ \toprule
\multicolumn{3}{r}{} & \multicolumn{8}{c}{Avg deviation from best OFV (\%)} & \multicolumn{8}{c}{Sum runtime over ten budgets (s)} \\ \cmidrule(rl){4-11} \cmidrule(rl){12-19}\
$n$ & $\Delta$ & t$^{\textup{cut}}$ & \textbf{QKBP} & RG & IHEA & LDP & DP & QK & Gurobi & Hexaly & \textbf{QKBP} & RG & IHEA & LDP & DP & QK & Gurobi & Hexaly \\
\midrule
1,000 & 25 & 0.10 & 0.21 & 0.02 & \bfseries 0.01 & \textemdash & \textemdash & \textemdash & 611.28 & 0.76 & \bfseries 0.47 & 231.2 & 113.7 & 1,200.0 & 1,200.0 & 1,200.0 & 1,163.3 & 1,200.0 \\
1,000 & 50 & 0.12 & 0.07 & 0.04 & \bfseries 0.00 & \textemdash & \textemdash & \textemdash & 583.47 & 2.03 & \bfseries 0.46 & 279.6 & 136.7 & 1,200.0 & 1,200.0 & 1,200.0 & 1,166.9 & 1,200.0 \\
1,000 & 75 & 0.14 & 0.23 & 0.06 & \bfseries 0.00 & \textemdash & \textemdash & \textemdash & 585.59 & 3.35 & \bfseries 0.66 & 255.0 & 156.2 & 1,200.0 & 1,200.0 & 1,200.0 & 1,215.8 & 1,200.0 \\
1,000 & 100 & 0.16 & 0.18 & 0.13 & \bfseries 0.00 & \textemdash & \textemdash & \textemdash & 582.96 & 4.21 & \bfseries 0.68 & 254.9 & 170.5 & 1,200.0 & 1,200.0 & 1,200.0 & 1,232.0 & 1,207.0 \\
2,000 & 25 & 0.27 & 0.04 & \bfseries 0.00 & \bfseries 0.00 & \textemdash & \textemdash & \textemdash & 634.27 & 3.24 & \bfseries 2.76 & 1,103.9 & 475.1 & 1,200.0 & 1,200.0 & 1,200.0 & 1,221.6 & 1,209.0 \\
2,000 & 50 & 0.38 & 0.07 & 0.05 & \bfseries 0.00 & \textemdash & \textemdash & \textemdash & 418.98 & 5.69 & \bfseries 2.87 & 1,144.5 & 504.9 & 1,200.0 & 1,200.0 & 1,200.0 & 1,219.8 & 1,220.0 \\
2,000 & 75 & 0.52 & 0.06 & 0.05 & \bfseries 0.00 & \textemdash & \textemdash & \textemdash & 294.63 & 7.39 & \bfseries 3.00 & 1,199.8 & 602.3 & 1,200.0 & 1,200.0 & 1,200.0 & 1,215.1 & 1,209.0 \\
2,000 & 100 & 0.65 & 0.09 & 0.08 & \bfseries 0.00 & \textemdash & \textemdash & \textemdash & 504.31 & 9.05 & \bfseries 3.01 & 1,152.5 & 675.3 & 1,200.0 & 1,200.0 & 1,200.0 & 1,223.6 & 1,200.0 \\
 \midrule Avg &  &  & 0.12 & 0.05 & \bfseries 0.00 & \textemdash & \textemdash & \textemdash & 526.94 & 4.46 & \bfseries 1.74 & 702.7 & 354.3 & 1,200.0 & 1,200.0 & 1,200.0 & 1,207.2 & 1,205.6 \\
Min &  &  & 0.04 & \bfseries 0.00 & \bfseries 0.00 & \textemdash & \textemdash & \textemdash & 294.63 & 0.76 & \bfseries 0.46 & 231.2 & 113.7 & 1,200.0 & 1,200.0 & 1,200.0 & 1,163.3 & 1,200.0 \\
Max &  &  & 0.23 & 0.13 & \bfseries 0.01 & \textemdash & \textemdash & \textemdash & 634.27 & 9.05 & \bfseries 3.01 & 1,199.8 & 675.3 & 1,200.0 & 1,200.0 & 1,200.0 & 1,232.0 & 1,220.0 \\
\bottomrule \\ [-1.5ex] \multicolumn{8}{l}{\textbf{QKBP} is the contribution in this paper}
\end{tabularx}

	\caption{Results for the instances of the \texttt{QKP-GroupII} collection: Each row represents ten graphs with the same values of $n$ and $\Delta$. The budget values are randomly selected in the interval $[50, \sum_{j=1}^{n} q_j]$. The abbreviation OFV stands for objective function value and the column $t^{\textup{cut}}$ reports the average running time of the simple parametric cut procedure in seconds. The time limit for each instance is 120 seconds. The hyphen (--) indicates that the respective approach did not find a solution for at least one of the ten instances within the time limit.}\label{tbl:qkpgroupII}	
\end{table*}
\begin{table*}
	\scriptsize
	\setlength{\tabcolsep}{5.3pt}
\begin{tabularx}{\textwidth}{rrr|rrrrrrrr|rrrrrrrr}
 \multicolumn{19}{c}{\texttt{QKP-GroupIII} collection} \\ \toprule
\multicolumn{3}{r}{} & \multicolumn{8}{c}{Avg deviation from best OFV (\%)} & \multicolumn{8}{c}{Sum runtime over five budgets (s)} \\ \cmidrule(rl){4-11} \cmidrule(rl){12-19}\
$n$ & $\Delta$ & t$^{\textup{cut}}$ & \textbf{QKBP} & RG & IHEA & LDP & DP & QK & Gurobi & Hexaly & \textbf{QKBP} & RG & IHEA & LDP & DP & QK & Gurobi & Hexaly \\
\midrule
5,000 & 25 & 1.27 & 0.03 & 0.02 & \bfseries 0.00 & \textemdash & \textemdash & \textemdash & \textemdash & 9.33 & \bfseries 9.93 & 602.4 & 727.4 & 600.0 & 600.0 & 600.0 & 602.5 & 600.0 \\
5,000 & 50 & 2.13 & 0.02 & 0.01 & \bfseries 0.00 & \textemdash & \textemdash & \textemdash & \textemdash & 12.99 & \bfseries 9.96 & 602.2 & 758.2 & 600.0 & 600.0 & 600.0 & 605.0 & 600.0 \\
5,000 & 75 & 3.16 & 0.04 & 0.03 & \bfseries 0.00 & \textemdash & \textemdash & \textemdash & \textemdash & 18.39 & \bfseries 10.78 & 601.1 & 785.1 & 600.0 & 600.0 & 600.0 & 608.7 & 600.0 \\
5,000 & 100 & 4.11 & 0.02 & 0.02 & \bfseries 0.00 & \textemdash & \textemdash & \textemdash & \textemdash & 24.54 & \bfseries 9.87 & 601.2 & 806.8 & 600.0 & 600.0 & 600.0 & 612.2 & 605.0 \\
6,000 & 25 & 1.74 & 0.05 & 0.03 & \bfseries 0.00 & \textemdash & \textemdash & \textemdash & \textemdash & 9.74 & \bfseries 14.94 & 603.1 & 789.0 & 600.0 & 600.0 & 600.0 & 603.8 & 600.0 \\
6,000 & 50 & 3.37 & 0.04 & 0.03 & \bfseries 0.00 & \textemdash & \textemdash & \textemdash & \textemdash & 15.53 & \bfseries 16.17 & 604.5 & 830.7 & 600.0 & 600.0 & 600.0 & 607.6 & 600.0 \\
6,000 & 75 & 5.12 & 0.05 & 0.04 & \bfseries 0.00 & \textemdash & \textemdash & \textemdash & \textemdash & 24.21 & \bfseries 14.17 & 602.8 & 862.5 & 600.0 & 600.0 & 600.0 & 611.4 & 600.0 \\
6,000 & 100 & 6.33 & 0.01 & 0.01 & \bfseries 0.00 & \textemdash & \textemdash & \textemdash & \textemdash & 39.70 & \bfseries 14.55 & 604.1 & 902.3 & 600.0 & 600.0 & 600.0 & 616.4 & 600.0 \\
 \midrule Avg &  &  & 0.03 & 0.02 & \bfseries 0.00 & \textemdash & \textemdash & \textemdash & \textemdash & 19.30 & \bfseries 12.55 & 602.7 & 807.7 & 600.0 & 600.0 & 600.0 & 608.5 & 600.6 \\
Min &  &  & 0.01 & 0.01 & \bfseries 0.00 & \textemdash & \textemdash & \textemdash & \textemdash & 9.33 & \bfseries 9.87 & 601.1 & 727.4 & 600.0 & 600.0 & 600.0 & 602.5 & 600.0 \\
Max &  &  & 0.05 & 0.04 & \bfseries 0.00 & \textemdash & \textemdash & \textemdash & \textemdash & 39.70 & \bfseries 16.17 & 604.5 & 902.3 & 600.0 & 600.0 & 600.0 & 616.4 & 605.0 \\
\bottomrule \\ [-1.5ex] \multicolumn{8}{l}{\textbf{QKBP} is the contribution in this paper}
\end{tabularx}

	\caption{Results for the instances of the \texttt{QKP-GroupIII} collection: Each row represents five graphs with the same values of $n$ and $\Delta$. The budget values are randomly selected in the interval $[50, \sum_{j=1}^{n} q_j]$. The abbreviation OFV stands for objective function value and the column $t^{\textup{cut}}$ reports the average running time of the simple parametric cut procedure in seconds. The time limit for each instance is 120 seconds. The hyphen (--) indicates that the respective approach did not find a solution for at least one of the five instances within the time limit.}\label{tbl:qkpgroupIII}	
\end{table*}

Tables~\ref{tbl:qkpgroupII} and \ref{tbl:qkpgroupIII} report the results for the collections \textbf{QKPGroupII} and \textbf{QKPGroupIII}, respectively. These collections contain much larger instances in terms of number of items ($n$) and thus can no longer be solved optimally by the exact approaches (QK, Gurobi, Hexaly). IHEA produces the best solutions closely followed by RG and QKBP. The dynamic programming-based approaches (LDP, DP, QK) do not find solutions within the time limit for any of these instances. QKBP is again orders of magnitudes faster while delivering almost the same solution quality as IHEA.

\begin{table*}
	\scriptsize
	\setlength{\tabcolsep}{4.6pt}
\begin{tabularx}{\textwidth}{rrr|rrrrrrrr|rrrrrrrr}
 \multicolumn{19}{c}{\texttt{Large-QKP} collection} \\ \toprule
\multicolumn{3}{r}{} & \multicolumn{8}{c}{Avg deviation from best OFV (\%)} & \multicolumn{8}{c}{Sum runtime over six budgets (s)} \\ \cmidrule(rl){4-11} \cmidrule(rl){12-19}\
$n$ & $\Delta$ & t$^{\textup{cut}}$ & \textbf{QKBP} & RG & IHEA & LDP & DP & QK & Gurobi & Hexaly & \textbf{QKBP} & RG & IHEA & LDP & DP & QK & Gurobi & Hexaly \\
\midrule
500 & 5 & 0.03 & 0.38 & 0.30 & 0.03 & \textemdash & 39.63 & \textemdash & \bfseries 0.00 & \bfseries 0.00 & \bfseries 0.03 & 20.0 & 33.9 & 720.0 & 286.0 & 720.0 & 9.2 & 294.0 \\
500 & 10 & 0.02 & 0.25 & 0.07 & \bfseries 0.00 & \textemdash & 94.30 & \textemdash & \bfseries 0.00 & 0.02 & \bfseries 0.03 & 21.4 & 36.5 & 720.0 & 238.7 & 720.0 & 27.9 & 585.0 \\
500 & 15 & 0.03 & 0.09 & 0.07 & \bfseries 0.00 & \textemdash & 150.03 & \textemdash & \bfseries 0.00 & 0.14 & \bfseries 0.04 & 21.0 & 39.0 & 720.0 & 222.6 & 720.0 & 248.3 & 720.0 \\
500 & 20 & 0.03 & 0.19 & 0.03 & \bfseries 0.00 & \textemdash & 195.72 & \textemdash & \bfseries 0.00 & 0.24 & \bfseries 0.04 & 21.4 & 40.7 & 720.0 & 207.4 & 720.0 & 491.6 & 720.0 \\
500 & 25 & 0.03 & 0.19 & 0.03 & \bfseries 0.00 & \textemdash & 266.17 & \textemdash & \bfseries 0.00 & 0.42 & \bfseries 0.03 & 21.6 & 35.1 & 720.0 & 207.1 & 720.0 & 435.1 & 720.0 \\
500 & 50 & 0.03 & 0.31 & 0.07 & \bfseries 0.00 & \textemdash & 296.93 & \textemdash & 0.09 & 0.92 & \bfseries 0.04 & 21.0 & 36.6 & 720.0 & 204.7 & 720.0 & 694.9 & 720.0 \\
500 & 75 & 0.03 & 0.44 & 0.13 & \bfseries 0.00 & \textemdash & 425.43 & \textemdash & 454.94 & 1.81 & \bfseries 0.04 & 20.8 & 46.0 & 720.0 & 178.0 & 720.0 & 720.4 & 720.0 \\
500 & 100 & 0.05 & 0.33 & 0.18 & \bfseries 0.00 & \textemdash & 463.68 & \textemdash & 256.09 & 2.06 & \bfseries 0.04 & 22.7 & 30.4 & 720.0 & 207.4 & 720.0 & 720.6 & 720.0 \\
1,000 & 5 & 0.08 & 0.13 & 0.14 & \bfseries 0.00 & \textemdash & \textemdash & \textemdash & 0.25 & 0.25 & \bfseries 0.10 & 135.0 & 74.8 & 720.0 & 720.0 & 720.0 & 583.8 & 720.0 \\
1,000 & 10 & 0.09 & 0.08 & 0.04 & \bfseries 0.00 & \textemdash & \textemdash & \textemdash & 3.02 & 0.34 & \bfseries 0.10 & 145.1 & 71.5 & 720.0 & 720.0 & 720.0 & 721.6 & 720.0 \\
1,000 & 15 & 0.09 & 0.11 & 0.04 & \bfseries 0.00 & \textemdash & \textemdash & \textemdash & 1.10 & 0.62 & \bfseries 0.11 & 151.9 & 85.3 & 720.0 & 720.0 & 720.0 & 720.4 & 720.0 \\
1,000 & 20 & 0.09 & 0.15 & 0.01 & \bfseries 0.00 & \textemdash & \textemdash & \textemdash & 691.78 & 0.73 & \bfseries 0.11 & 144.3 & 72.8 & 720.0 & 720.0 & 720.0 & 720.4 & 720.0 \\
1,000 & 25 & 0.11 & 0.04 & 0.01 & \bfseries 0.00 & \textemdash & \textemdash & \textemdash & 753.10 & 1.07 & \bfseries 0.12 & 135.8 & 84.7 & 720.0 & 720.0 & 720.0 & 720.8 & 720.0 \\
1,000 & 50 & 0.11 & 0.14 & \bfseries 0.00 & \bfseries 0.00 & \textemdash & \textemdash & \textemdash & 878.06 & 2.09 & \bfseries 0.11 & 140.4 & 91.6 & 720.0 & 720.0 & 720.0 & 721.1 & 720.0 \\
2,000 & 5 & 0.20 & 0.11 & 0.05 & \bfseries 0.00 & \textemdash & \textemdash & \textemdash & 509.97 & 0.57 & \bfseries 0.49 & 610.8 & 251.3 & 720.0 & 720.0 & 720.0 & 720.5 & 720.0 \\
2,000 & 10 & 0.20 & 0.08 & 0.01 & \bfseries 0.00 & \textemdash & \textemdash & \textemdash & 826.14 & 1.01 & \bfseries 0.49 & 605.1 & 260.6 & 720.0 & 720.0 & 720.0 & 728.3 & 720.0 \\
2,000 & 15 & 0.24 & 0.06 & 0.01 & \bfseries 0.00 & \textemdash & \textemdash & \textemdash & 1,051.86 & 1.72 & \bfseries 0.50 & 616.8 & 301.4 & 720.0 & 720.0 & 720.0 & 738.6 & 720.0 \\
2,000 & 20 & 0.25 & 0.09 & 0.03 & \bfseries 0.00 & \textemdash & \textemdash & \textemdash & 844.24 & 2.41 & \bfseries 0.50 & 616.5 & 309.0 & 720.0 & 720.0 & 720.0 & 738.5 & 721.0 \\
2,000 & 25 & 0.26 & 0.07 & \bfseries 0.00 & \bfseries 0.00 & \textemdash & \textemdash & \textemdash & 1,147.67 & 2.83 & \bfseries 0.51 & 607.3 & 309.5 & 720.0 & 720.0 & 720.0 & 733.3 & 725.0 \\
5,000 & 5 & 0.59 & 0.03 & 0.01 & \bfseries 0.00 & \textemdash & \textemdash & \textemdash & 950.65 & 2.82 & \bfseries 3.37 & 722.5 & 840.3 & 720.0 & 720.0 & 720.0 & 722.4 & 726.0 \\
5,000 & 10 & 0.78 & 0.01 & 0.01 & \bfseries 0.00 & \textemdash & \textemdash & \textemdash & 1,061.24 & 5.28 & \bfseries 3.50 & 722.0 & 854.2 & 720.0 & 720.0 & 720.0 & 727.5 & 738.0 \\
5,000 & 15 & 1.06 & 0.03 & 0.02 & \bfseries 0.00 & \textemdash & \textemdash & \textemdash & \textemdash & 6.96 & \bfseries 3.48 & 721.4 & 863.6 & 720.0 & 720.0 & 720.0 & 722.5 & 729.0 \\
5,000 & 20 & 1.09 & 0.03 & 0.03 & \bfseries 0.00 & \textemdash & \textemdash & \textemdash & \textemdash & 7.55 & \bfseries 3.37 & 721.4 & 875.6 & 720.0 & 720.0 & 720.0 & 722.5 & 720.0 \\
10,000 & 5 & 1.53 & \bfseries 0.00 & 0.01 & \bfseries 0.00 & \textemdash & \textemdash & \textemdash & \textemdash & 7.06 & \bfseries 15.10 & 729.8 & 1,228.6 & 720.0 & 720.0 & 720.0 & 722.9 & 720.0 \\
 \midrule Avg &  &  & 0.14 & 0.05 & \bfseries 0.00 & \textemdash & \textemdash & \textemdash & \textemdash & 2.04 & \bfseries 1.34 & 320.7 & 286.4 & 720.0 & 553.0 & 720.0 & 617.2 & 698.2 \\
Min &  &  & \bfseries 0.00 & \bfseries 0.00 & \bfseries 0.00 & \textemdash & 39.63 & \textemdash & \bfseries 0.00 & \bfseries 0.00 & \bfseries 0.03 & 20.0 & 30.4 & 720.0 & 178.0 & 720.0 & 9.2 & 294.0 \\
Max &  &  & 0.44 & 0.30 & \bfseries 0.03 & \textemdash & \textemdash & \textemdash & \textemdash & 7.55 & \bfseries 15.10 & 729.8 & 1,228.6 & 720.0 & 720.0 & 720.0 & 738.6 & 738.0 \\
\bottomrule \\ [-1.5ex] \multicolumn{8}{l}{\textbf{QKBP} is the contribution in this paper}
\end{tabularx}

	\caption{Results for the instances of the \texttt{Large-QKP} collection: Each row represents a graph with a specific number of nodes $n$ and a density $\Delta$. For each graph, there are six instances with different budget values. These budget values are chosen as fractions of the total node weights of the graph. The fractions used are 0.025, 0.05, 0.1, 0.25, 0.5, and 0.75. The abbreviation OFV stands for objective function value and the column $t^{\textup{cut}}$ reports the running time of the simple parametric cut procedure in seconds. Note that the simple parametric cut procedure is applied only once for all six instances of the same graph. The time limit for each instance is 120 seconds. The hyphen (--) indicates that the respective approach did not find a solution for at least one of the six instances within the time limit.}\label{tbl:new_qkp}
\end{table*}

Next, we present the results obtained for the collection \textbf{Large-QKP} which contains 144 medium and large instances with up to $n=10,000$ nodes. Table~\ref{tbl:new_qkp} reports for each approach and each combination of number of nodes $n$ and graph density $\Delta$ the average deviation from the best OFV in percent and the total running time in seconds. This time each row of the table represents a graph with a specific number of nodes $n$ and a density $\Delta$. For each graph, there are six instances with different budget values. These budget values are chosen as fractions of the total node weights of the graph. The fractions used are 0.025, 0.05, 0.1, 0.25, 0.5, 0.75. The results demonstrate the scalability of our QKBP approach. It is considerably faster than all other approaches and still delivers high-quality solutions, with average deviations from the best OFV often below 0.2\%. The RG and IHEA approaches perform very well in terms of solution quality but they are considerably slower compared to the breakpoints approach QKBP. The good performance of the RG approach for these instances confirms the finding of \cite{schauer2016asymptotic} that for the standard QKP instances generated with the procedure of \cite{gallo1980quadratic}, even simple heuristics may determine solutions whose objective value are (asypmtotically) close to the optimal value. Among the exact approaches (QK, Gurobi, and Hexaly), Hexaly scales the best but cannot compete with the best heuristics both in terms of solution quality and running time. The dynamic programming-based approaches (LDP, DP and QK) do not scale well. Note that in other papers the approaches LDP, DP and QK were successfully applied to instances with up to $n=1{,}500$ nodes but in those papers a higher time limit was employed and the budget fraction was not varied systematically. It turns out that the performance of the dynamic programming-based approaches strongly depends on the budget fraction. Low budget fractions turn out to be more challenging. The IHEA algorithm exceeded the time limit for some large instances. This happened because we only check the elapsed time at the end of each iteration. For large instances, the duration of each iteration increases, leading to this issue.

We proceed to present the results obtained for the collection \textbf{Dispersion-QKP}, which contains 576 instances with up to $n=2,000$ nodes. These instances are divided into four groups (\texttt{geo}, \texttt{wgeo}, \texttt{expo}, and \texttt{ran}) based on the strategy used to generate the edge weights (see Section~\ref{sec:problem_instances}). The instances of the groups \texttt{geo} and \texttt{wgeo} do not belong to the family of ``easy" instances defined in \cite{schauer2016asymptotic}, since the edge weights are not chosen independently. However, the results turn out to be similar for all four groups. Therefore, we report here only the results for the strategy \texttt{geo} in Table~\ref{tbl:dispersion_qkp_geo}. The results for the other three strategies are reported in the appendix in Tables~\ref{tbl:dispersion_qkp_wgeo}--\ref{tbl:dispersion_qkp_ran}. The structure of the four tables is identical to that of Table~\ref{tbl:new_qkp}, i.e., each row reports the results for six instances that have the same underlying graph but differ with respect to the budget values. 

\begin{table*}
	\scriptsize
	\setlength{\tabcolsep}{4.9pt}
\begin{tabularx}{\textwidth}{rrr|rrrrrrrr|rrrrrrrr}
 \multicolumn{19}{c}{\texttt{Dispersion-QKP (geo)} collection} \\ \toprule
\multicolumn{3}{r}{} & \multicolumn{8}{c}{Avg deviation from best OFV (\%)} & \multicolumn{8}{c}{Sum runtime over six budgets (s)} \\ \cmidrule(rl){4-11} \cmidrule(rl){12-19}\
$n$ & $\Delta$ & t$^{\textup{cut}}$ & \textbf{QKBP} & RG & IHEA & LDP & DP & QK & Gurobi & Hexaly & \textbf{QKBP} & RG & IHEA & LDP & DP & QK & Gurobi & Hexaly \\
\midrule
300 & 5 & 0.02 & 0.64 & 0.45 & 0.03 & \textemdash & 32.56 & \textemdash & \bfseries 0.00 & \bfseries 0.00 & \bfseries 0.02 & 5.7 & 21.8 & 720.0 & 53.6 & 482.8 & 5.0 & 84.0 \\
300 & 10 & 0.02 & 0.29 & 0.20 & \bfseries 0.00 & \textemdash & 108.12 & \textemdash & \bfseries 0.00 & \bfseries 0.00 & \bfseries 0.02 & 6.0 & 23.7 & 720.0 & 46.5 & 483.7 & 10.6 & 196.0 \\
300 & 25 & 0.02 & 0.12 & 0.05 & \bfseries 0.00 & \textemdash & 196.31 & \textemdash & \bfseries 0.00 & \bfseries 0.00 & \bfseries 0.02 & 6.1 & 25.2 & 720.0 & 41.8 & 720.0 & 58.4 & 547.0 \\
300 & 50 & 0.02 & 0.29 & 0.02 & \bfseries 0.00 & \textemdash & 224.26 & \textemdash & \bfseries 0.00 & 0.13 & \bfseries 0.02 & 6.1 & 25.5 & 720.0 & 41.0 & 720.0 & 184.6 & 659.0 \\
300 & 75 & 0.02 & 0.38 & 0.05 & \bfseries 0.00 & \textemdash & 260.68 & \textemdash & \bfseries 0.00 & 0.38 & \bfseries 0.02 & 6.1 & 25.8 & 720.0 & 41.1 & 720.0 & 303.1 & 683.0 \\
300 & 100 & 0.03 & 0.64 & 0.02 & \bfseries 0.00 & \textemdash & 314.22 & \textemdash & \bfseries 0.00 & 0.74 & \bfseries 0.02 & 6.2 & 23.9 & 720.0 & 40.6 & 368.1 & 288.3 & 720.0 \\
500 & 5 & 0.03 & 0.25 & 0.34 & \bfseries 0.00 & \textemdash & 66.38 & \textemdash & \bfseries 0.00 & \bfseries 0.00 & \bfseries 0.03 & 20.2 & 35.1 & 720.0 & 328.0 & 720.0 & 13.0 & 479.0 \\
500 & 10 & 0.03 & 0.16 & 0.11 & \bfseries 0.00 & \textemdash & 139.84 & \textemdash & \bfseries 0.00 & 0.08 & \bfseries 0.04 & 20.0 & 44.3 & 720.0 & 273.6 & 720.0 & 104.8 & 712.0 \\
500 & 25 & 0.03 & 0.11 & 0.05 & \bfseries 0.00 & \textemdash & 211.78 & \textemdash & 0.74 & 0.14 & \bfseries 0.04 & 20.3 & 34.0 & 720.0 & 251.1 & 720.0 & 542.1 & 720.0 \\
500 & 50 & 0.03 & 0.18 & 0.04 & \bfseries 0.00 & \textemdash & 316.61 & \textemdash & 1.19 & 0.31 & \bfseries 0.04 & 20.9 & 33.2 & 720.0 & 238.1 & 720.0 & 691.1 & 720.0 \\
500 & 75 & 0.05 & 0.36 & 0.13 & \bfseries 0.00 & \textemdash & 342.80 & \textemdash & 387.88 & 0.78 & \bfseries 0.04 & 20.4 & 37.0 & 720.0 & 231.0 & 720.0 & 720.5 & 720.0 \\
500 & 100 & 0.03 & 0.36 & 0.07 & \bfseries 0.00 & \textemdash & 356.46 & \textemdash & 857.86 & 0.40 & \bfseries 0.04 & 20.7 & 38.0 & 720.0 & 229.1 & 720.0 & 720.5 & 720.0 \\
1,000 & 5 & 0.06 & 0.12 & 0.15 & \bfseries 0.00 & \textemdash & \textemdash & \textemdash & \bfseries 0.00 & 0.11 & \bfseries 0.09 & 132.3 & 72.3 & 720.0 & 720.0 & 720.0 & 373.9 & 720.0 \\
1,000 & 10 & 0.08 & 0.04 & 0.02 & \bfseries 0.00 & \textemdash & \textemdash & \textemdash & 0.01 & 0.25 & \bfseries 0.10 & 135.2 & 67.0 & 720.0 & 720.0 & 720.0 & 580.1 & 720.0 \\
1,000 & 25 & 0.09 & 0.11 & 0.03 & \bfseries 0.00 & \textemdash & \textemdash & \textemdash & 795.51 & 0.33 & \bfseries 0.11 & 134.9 & 72.7 & 720.0 & 720.0 & 720.0 & 624.5 & 720.0 \\
1,000 & 50 & 0.11 & 0.15 & 0.01 & \bfseries 0.00 & \textemdash & \textemdash & \textemdash & 832.32 & 0.75 & \bfseries 0.11 & 139.1 & 69.2 & 720.0 & 720.0 & 720.0 & 669.8 & 720.0 \\
1,000 & 75 & 0.12 & 0.15 & 0.02 & \bfseries 0.00 & \textemdash & \textemdash & \textemdash & 1,063.30 & 0.96 & \bfseries 0.11 & 138.6 & 78.7 & 720.0 & 720.0 & 720.0 & 722.5 & 721.0 \\
1,000 & 100 & 0.16 & 0.28 & 0.09 & \bfseries 0.00 & \textemdash & \textemdash & \textemdash & 746.14 & 1.62 & \bfseries 0.11 & 137.4 & 98.7 & 720.0 & 720.0 & 720.0 & 722.8 & 724.0 \\
2,000 & 5 & 0.19 & 0.06 & 0.02 & \bfseries 0.00 & \textemdash & \textemdash & \textemdash & 290.65 & 0.24 & \bfseries 0.44 & 601.7 & 216.0 & 720.0 & 720.0 & 720.0 & 647.2 & 720.0 \\
2,000 & 10 & 0.20 & 0.03 & 0.01 & \bfseries 0.00 & \textemdash & \textemdash & \textemdash & 874.80 & 0.50 & \bfseries 0.47 & 610.1 & 238.2 & 720.0 & 720.0 & 720.0 & 724.6 & 720.0 \\
2,000 & 25 & 0.25 & 0.04 & \bfseries 0.00 & \bfseries 0.00 & \textemdash & \textemdash & \textemdash & 1,125.57 & 1.40 & \bfseries 0.53 & 612.6 & 280.8 & 720.0 & 720.0 & 720.0 & 728.9 & 723.0 \\
2,000 & 50 & 0.34 & 0.07 & \bfseries 0.00 & \bfseries 0.00 & \textemdash & \textemdash & \textemdash & 1,224.72 & 2.88 & \bfseries 0.47 & 606.6 & 295.6 & 720.0 & 720.0 & 720.0 & 725.6 & 732.0 \\
2,000 & 75 & 0.42 & 0.06 & 0.01 & \bfseries 0.00 & \textemdash & \textemdash & \textemdash & 951.39 & 3.96 & \bfseries 0.46 & 612.7 & 279.2 & 720.0 & 720.0 & 720.0 & 728.5 & 720.0 \\
2,000 & 100 & 0.55 & 0.15 & 0.02 & \bfseries 0.00 & \textemdash & \textemdash & \textemdash & 1,067.38 & 5.09 & \bfseries 0.47 & 611.4 & 280.4 & 720.0 & 720.0 & 720.0 & 730.8 & 720.0 \\
 \midrule Avg &  &  & 0.21 & 0.08 & \bfseries 0.00 & \textemdash & \textemdash & \textemdash & 425.81 & 0.88 & \bfseries 0.16 & 193.0 & 100.7 & 720.0 & 435.6 & 685.6 & 484.2 & 650.8 \\
Min &  &  & 0.03 & \bfseries 0.00 & \bfseries 0.00 & \textemdash & 32.56 & \textemdash & \bfseries 0.00 & \bfseries 0.00 & \bfseries 0.02 & 5.7 & 21.8 & 720.0 & 40.6 & 368.1 & 5.0 & 84.0 \\
Max &  &  & 0.64 & 0.45 & \bfseries 0.03 & \textemdash & \textemdash & \textemdash & 1,224.72 & 5.09 & \bfseries 0.53 & 612.7 & 295.6 & 720.0 & 720.0 & 720.0 & 730.8 & 732.0 \\
\bottomrule \\ [-1.5ex] \multicolumn{8}{l}{\textbf{QKBP} is the contribution in this paper}
\end{tabularx}

	\caption{Results for the instances of the \texttt{Dispersion-QKP} collection with distance computation strategy \texttt{geo}: Each row represents a graph with a certain number of nodes $n$ and a density $\Delta$. For each graph, there are six instances with different budget values. These budget values are chosen as fractions of the total node weights of the graph. The fractions used are 0.025, 0.05, 0.1, 0.25, 0.5, and 0.75. The abbreviation OFV stands for objective function value and the column $t^{\textup{cut}}$ reports the running time of the simple parametric cut procedure in seconds. Note that the simple parametric cut procedure is applied only once for all six instances of the same graph. The time limit for each instance is 120 seconds. The hyphen (--) indicates that the respective approach did not find a solution for at least one of the six instances within the time limit.}\label{tbl:dispersion_qkp_geo}
\end{table*}

The main conclusion that can be drawn from the four tables is that our breakpoints approach is, by orders of magnitude, the fastest algorithm, and  consistently produces solutions that are very close to the best solutions found by any of the other approaches. The following additional conclusions can be drawn. For larger instances, the approaches differ in terms of scalability. The dynamic programming-based approaches LDP, DP and QK often fail to find solutions within the time limit. The Gurobi-based approach starts to produce low-quality solutions for instances with $n\geq 1,000$ and $\Delta\geq 25$. The only exact approach that finds reasonably good solutions for all instances within the time limit of 120 seconds is the Hexaly-based approach. However, for the largest instances with $\Delta=100$, the average deviations from the best objective values increase to over 5\%. The IHEA approach has the lowest average deviation from the best objective values closely followed by RG and our proposed breakpoints approach. The good performance of the greedy approach is surprising for the instances of the groups \texttt{geo} and \texttt{wgeo}, since they do not belong to the family of ``easy" instances defined by \cite{schauer2016asymptotic}.

\begin{table*}
	\scriptsize
	\setlength{\tabcolsep}{5.2pt}
\begin{tabularx}{\textwidth}{Xrrr|rrrrrrrr|rrrrr}
 \multicolumn{17}{c}{\texttt{TeamFormation-QKP-1} collection} \\ \toprule
\multicolumn{4}{r}{} & \multicolumn{8}{c}{Avg deviation from best OFV (\%)} & \multicolumn{5}{c}{Sum runtime over six budgets (s)} \\ \cmidrule(rl){5-12} \cmidrule(rl){13-17}\
 & $n$ & $\Delta$ & t$^{\textup{cut}}$ & \textbf{QKBP} & RG & IHEA & LDP & DP & QK & Gurobi & Hexaly & \textbf{QKBP} & RG & IHEA & Gurobi & Hexaly \\
\midrule
IMDB & 1,021 & 2.15 & 0.83 & 0.36 & 19.13 & 1.89 & \textemdash & \textemdash & \textemdash & \bfseries 0.00 & \bfseries 0.00 & \bfseries 0.04 & 91.5 & 61.1 & 14.0 & 104.0 \\
DBLP & 7,159 & 0.06 & 0.62 & 0.01 & 52.71 & 49.12 & \textemdash & \textemdash & \textemdash & \bfseries 0.00 & \bfseries 0.00 & \bfseries 3.40 & 723.1 & 926.0 & 8.3 & 86.0 \\
StackOverflow & 8,834 & 0.16 & 0.09 & 0.01 & 0.87 & 0.77 & \textemdash & \textemdash & \textemdash & \bfseries 0.00 & 0.31 & \bfseries 5.81 & 724.8 & 1,040.7 & 103.0 & 587.0 \\
Bibsonomy & 9,269 & 0.07 & 0.80 & \bfseries 0.00 & 81.07 & 80.02 & \textemdash & \textemdash & \textemdash & \bfseries 0.00 & 1.24 & \bfseries 6.32 & 723.7 & 1,085.1 & 21.2 & 394.0 \\
Synthetic\_01 & 7,000 & 0.15 & 0.66 & 0.02 & 45.72 & 41.39 & \textemdash & \textemdash & \textemdash & \bfseries 0.00 & 0.36 & \bfseries 3.15 & 722.9 & 917.8 & 35.9 & 311.0 \\
Synthetic\_02 & 7,000 & 0.15 & 0.61 & 0.02 & 59.72 & 47.02 & \textemdash & \textemdash & \textemdash & \bfseries 0.00 & 0.07 & \bfseries 3.41 & 722.1 & 929.0 & 43.6 & 418.0 \\
Synthetic\_03 & 7,000 & 0.15 & 0.62 & 0.01 & 57.81 & 44.30 & \textemdash & \textemdash & \textemdash & \bfseries 0.00 & 0.03 & \bfseries 3.39 & 723.7 & 921.6 & 44.9 & 349.0 \\
Synthetic\_04 & 7,000 & 0.14 & 0.62 & \bfseries 0.00 & 46.06 & 26.94 & \textemdash & \textemdash & \textemdash & \bfseries 0.00 & 0.69 & \bfseries 3.19 & 727.0 & 930.8 & 44.6 & 221.0 \\
Synthetic\_05 & 7,000 & 0.15 & 0.62 & 0.01 & 55.35 & 38.15 & \textemdash & \textemdash & \textemdash & \bfseries 0.00 & 0.60 & \bfseries 3.22 & 723.0 & 922.7 & 32.6 & 554.0 \\
Synthetic\_06 & 7,000 & 0.14 & 0.61 & 0.02 & 49.58 & 33.62 & \textemdash & \textemdash & \textemdash & \bfseries 0.00 & 0.87 & \bfseries 3.45 & 722.6 & 923.6 & 42.4 & 325.0 \\
Synthetic\_07 & 7,000 & 0.15 & 0.61 & 0.01 & 64.22 & 56.24 & \textemdash & \textemdash & \textemdash & \bfseries 0.00 & 0.76 & \bfseries 3.40 & 723.3 & 928.8 & 32.5 & 491.0 \\
Synthetic\_08 & 7,000 & 0.16 & 0.61 & 0.01 & 42.61 & 40.92 & \textemdash & \textemdash & \textemdash & \bfseries 0.00 & 3.88 & \bfseries 3.33 & 722.6 & 922.4 & 34.0 & 438.0 \\
Synthetic\_09 & 7,000 & 0.14 & 0.62 & 0.02 & 54.75 & 40.07 & \textemdash & \textemdash & \textemdash & \bfseries 0.00 & 1.32 & \bfseries 3.74 & 721.7 & 921.1 & 36.6 & 402.0 \\
Synthetic\_10 & 7,000 & 0.15 & 0.64 & 0.01 & 51.40 & 35.13 & \textemdash & \textemdash & \textemdash & \bfseries 0.00 & 0.09 & \bfseries 3.13 & 724.6 & 924.0 & 40.2 & 422.0 \\
 \midrule Avg &  &  &  & 0.04 & 48.64 & 38.26 & \textemdash & \textemdash & \textemdash & \bfseries 0.00 & 0.73 & \bfseries 3.50 & 678.3 & 882.5 & 38.1 & 364.4 \\
Min &  &  &  & \bfseries 0.00 & 0.87 & 0.77 & \textemdash & \textemdash & \textemdash & \bfseries 0.00 & \bfseries 0.00 & \bfseries 0.04 & 91.5 & 61.1 & 8.3 & 86.0 \\
Max &  &  &  & 0.36 & 81.07 & 80.02 & \textemdash & \textemdash & \textemdash & \bfseries 0.00 & 3.88 & \bfseries 6.32 & 727.0 & 1,085.1 & 103.0 & 587.0 \\
\bottomrule \\ [-1.5ex] \multicolumn{8}{l}{\textbf{QKBP} is the contribution in this paper}
\end{tabularx}

	\caption{Results for the instances of the \texttt{TeamFormation-QKP-1} collection: Each row represents a graph with a certain number of nodes $n$ and a density $\Delta$. For each graph, there are six instances with different budget values. These budget values are chosen as fractions of the total node weights of the graph. The fractions used are 0.025, 0.05, 0.1, 0.25, 0.5, and 0.75. The abbreviation OFV stands for objective function value and the column $t^{\textup{cut}}$ reports the running time of the simple parametric cut procedure in seconds. Note that the simple parametric cut procedure is applied only once for all six instances of the same graph. The time limit for each instance is 120 seconds. The hyphen (--) indicates that the respective approach did not find a solution for at least one of the six instances within the time limit.}\label{tbl:team_formation_qkp1}
\end{table*}

Next, we present the results for the collection \textbf{TeamFormation-QKP-1}, which contains 84 instances with up to $n=9,269$ nodes. The aggregated results for these instances are presented in Table~\ref{tbl:team_formation_qkp1}. Each row of the table represents a graph. For each graph, there are six instances with different budget values. These budget values are chosen as fractions of the total node weights of the graph. The fractions used are 0.025, 0.05, 0.1, 0.25, 0.5, 0.75. The first four rows (IMDB, DBLP, StackOverflow, and Bibsonomy) represent real-world graphs and the remaining rows represent synthetically generated graphs. Even though these instances contain many nodes, the Gurobi-based approach finds optimal solutions for almost all instances. This is because the size of the optimization problem solved by Gurobi, in terms of the number of decision variables and constraints, depends primarily on the number of edges, and since the density of the graphs $\Delta$ is extremely low, often below 0.2\%, the resulting models are rather small. The availability of optimal solutions allows us to conclude that the QKBP approach is capable of finding optimal and near-optimal solutions for these instances. An interesting observation is that the approaches IHEA and RG deliver low-quality solutions with up to 80\% deviation from the best objective value. We explain this poor performance by the fact that the underlying graphs consist of many disconnected subgraphs. The RG approach starts in one such subgraph and iteratively adds its nodes. When all nodes of the subgraph have been selected and the budget has not yet been reached, the greedy algorithm makes an arbitrary decision to add the next node because all remaining nodes are equally attractive to include, i.e., none of these nodes would increase the total utility because they are not connected to any of the selected nodes and all singleton utilities are zero in these team-formation instances. The IHEA approach has the same limitation because it constructs an initial solution using such a greedy search strategy. The QKBP approach does not suffer from this limitation for these instances, because it finds many breakpoints (see Figure~\ref{fig:breakpoints_DBLP}), and thus uses the greedy subroutine only to add or remove a few nodes. These results demonstrate that greedy search strategies are not able to achieve high-quality solutions for some types of instances.

\begin{figure*}
	\includegraphics[width=\textwidth]{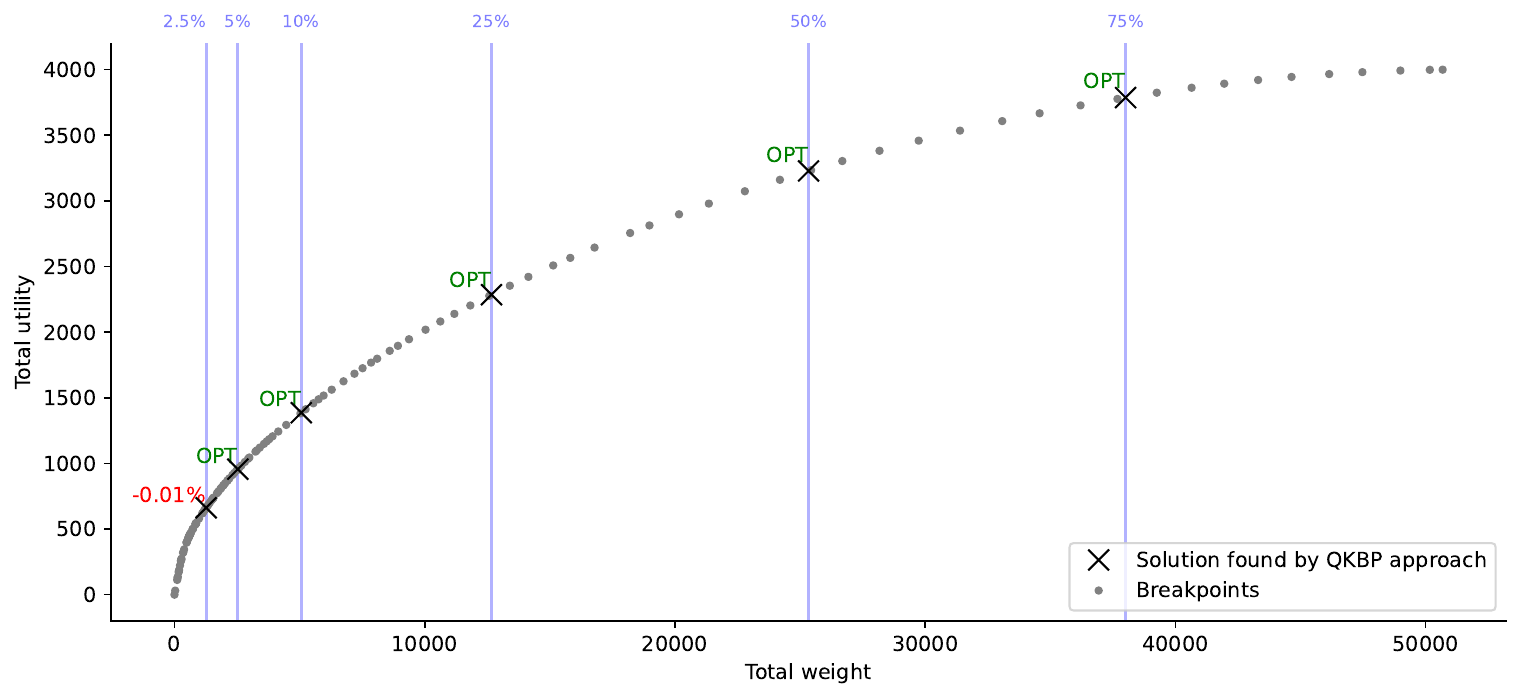}
	\caption{Bibsonomy data set: breakpoints and solutions for budget fractions 2.5\%, 5\%, 10\%, \ldots, 75\%. The solution provided by the  QKBP method for budget fraction 2.5\% deviates from the optimal solution by 0.01\%, all other solutions are optimal.}\label{fig:breakpoints_DBLP}
\end{figure*}

\begin{table*}
	\scriptsize
	\setlength{\tabcolsep}{4.8pt}
\begin{tabularx}{\textwidth}{rrr|rrrrrrrr|rrrrrrrr}
 \multicolumn{19}{c}{\texttt{TeamFormation-QKP-2} collection} \\ \toprule
\multicolumn{3}{r}{} & \multicolumn{8}{c}{Avg deviation from best OFV (\%)} & \multicolumn{8}{c}{Sum runtime over six budgets (s)} \\ \cmidrule(rl){4-11} \cmidrule(rl){12-19}\
$n$ & $\Delta$ & t$^{\textup{cut}}$ & \textbf{QKBP} & RG & IHEA & LDP & DP & QK & Gurobi & Hexaly & \textbf{QKBP} & RG & IHEA & LDP & DP & QK & Gurobi & Hexaly \\
\midrule
1,000 & 13.43 & 0.09 & 0.02 & 17.54 & 4.94 & \textemdash & \textemdash & \textemdash & \bfseries 0.00 & 0.76 & \bfseries 0.05 & 104.2 & 50.0 & 720.0 & 720.0 & 720.0 & 121.5 & 720.0 \\
2,000 & 12.38 & 0.24 & \bfseries 0.08 & 12.48 & 6.72 & \textemdash & \textemdash & \textemdash & 52.57 & 4.24 & \bfseries 0.26 & 453.5 & 179.6 & 720.0 & 720.0 & 720.0 & 718.3 & 720.0 \\
4,000 & 12.68 & 0.56 & \bfseries 0.00 & 14.76 & 10.09 & \textemdash & \textemdash & \textemdash & 573.14 & 4.15 & \bfseries 1.38 & 721.5 & 717.8 & 720.0 & 720.0 & 720.0 & 724.3 & 732.0 \\
6,000 & 12.54 & 1.08 & \bfseries 0.00 & 8.56 & 5.15 & \textemdash & \textemdash & \textemdash & \textemdash & 7.06 & \bfseries 3.22 & 722.9 & 915.4 & 720.0 & 720.0 & 720.0 & 729.9 & 720.0 \\
8,000 & 12.57 & 1.80 & \bfseries 0.00 & 10.73 & 7.60 & \textemdash & \textemdash & \textemdash & \textemdash & 9.86 & \bfseries 6.37 & 723.9 & 1,073.0 & 720.0 & 720.0 & 720.0 & 724.0 & 720.0 \\
10,000 & 12.77 & 2.86 & \bfseries 0.00 & 10.10 & 7.19 & \textemdash & \textemdash & \textemdash & \textemdash & 18.23 & \bfseries 11.27 & 733.2 & 1,255.0 & 720.0 & 720.0 & 720.0 & 726.4 & 720.0 \\
 \midrule Avg &  &  & \bfseries 0.02 & 12.36 & 6.95 & \textemdash & \textemdash & \textemdash & \textemdash & 7.38 & \bfseries 3.76 & 576.5 & 698.5 & 720.0 & 720.0 & 720.0 & 624.1 & 722.0 \\
Min &  &  & \bfseries 0.00 & 8.56 & 4.94 & \textemdash & \textemdash & \textemdash & \bfseries 0.00 & 0.76 & \bfseries 0.05 & 104.2 & 50.0 & 720.0 & 720.0 & 720.0 & 121.5 & 720.0 \\
Max &  &  & \bfseries 0.08 & 17.54 & 10.09 & \textemdash & \textemdash & \textemdash & \textemdash & 18.23 & \bfseries 11.27 & 733.2 & 1,255.0 & 720.0 & 720.0 & 720.0 & 729.9 & 732.0 \\
\bottomrule \\ [-1.5ex] \multicolumn{8}{l}{\textbf{QKBP} is the contribution in this paper}
\end{tabularx}

	\caption{Results for the instances of the \texttt{TeamFormation-QKP-2} collection: Each row represents a graph with a certain number of nodes $n$ and a density $\Delta$. For each graph, there are six instances with different budget values. These budget values are chosen as fractions of the total node weights of the graph. The fractions used are 0.025, 0.05, 0.1, 0.25, 0.5, and 0.75. The abbreviation OFV stands for objective function value and the column $t^{\textup{cut}}$ reports the running time of the simple parametric cut procedure in seconds. Note that the simple parametric cut procedure is applied only once for all six instances of the same graph. The time limit for each instance is 120 seconds. The hyphen (--) indicates that the respective approach did not find a solution for at least one of the six instances within the time limit.}\label{tbl:team_formation_qkp2}
\end{table*}

\begin{figure*}
	\includegraphics[width=\textwidth]{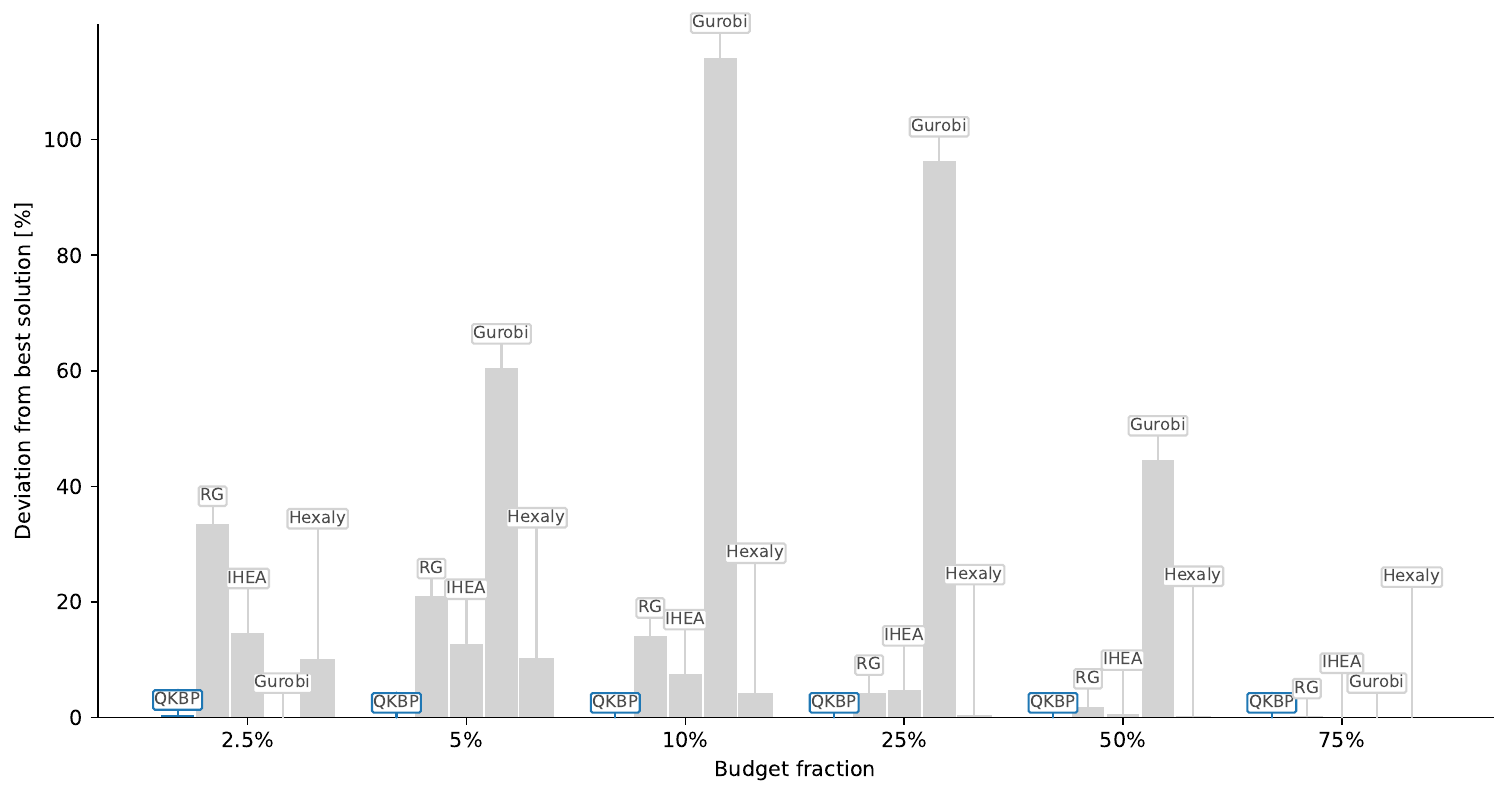}
	\caption{Graph with $n=2,000$ nodes of the collection TeamFormation-QKP-2: the bars represent the deviations from the best objective function value in percent for budget fractions 2.5\%, 5\%, 10\%, 25\%, 50\%, and 75\% obtained with a time limit of 120 seconds. For the QKBP approach, only the first bar is visible because for all other budget fractions it obtained the best solution and hence has deviations of 0\%. We did not include the dynamic programming-based approaches (LDP, DP, QK) as they were not able to find solutions within the time limit.}
	\label{fig:deviations_tf_2000}
\end{figure*}

\begin{table*}
	\scriptsize
	\setlength{\tabcolsep}{6.6pt}
\begin{tabularx}{\textwidth}{rrr|rrrrrrrr|rrrrr}
 \multicolumn{16}{c}{\texttt{TeamFormation-QKP-2} collection} \\ \toprule
\multicolumn{3}{r}{} & \multicolumn{8}{c}{Avg deviation from best OFV (\%)} & \multicolumn{5}{c}{Sum running time (s)} \\ \cmidrule(rl){4-11} \cmidrule(rl){12-16}\
$n$ & $\Delta$ & t$^{\textup{cut}}$ & \textbf{QKBP} & RG & IHEA & LDP & DP & QK & Gurobi & Hexaly & \textbf{QKBP} & RG & IHEA & Gurobi & Hexaly \\
\midrule
1,000 & 13.43 & 0.09 & 0.02 & 17.54 & 4.94 & \textemdash & 3.57 & \textemdash & \bfseries 0.00 & 0.02 & \bfseries 0.04 & 108.7 & 48.4 & 107.6 & 9,508.0 \\
2,000 & 12.38 & 0.22 & 0.20 & 12.60 & 6.84 & \textemdash & \textemdash & \textemdash & \bfseries 0.00 & 1.92 & \bfseries 0.25 & 735.8 & 172.7 & 1,970.1 & 21,600.0 \\
4,000 & 12.68 & 0.56 & \bfseries 0.01 & 14.74 & 9.86 & \textemdash & \textemdash & \textemdash & 40.96 & 0.96 & \bfseries 1.17 & 6,328.7 & 958.5 & 16,564.8 & 21,612.0 \\
6,000 & 12.54 & 1.08 & \bfseries 0.03 & 8.58 & 5.04 & \textemdash & \textemdash & \textemdash & 268.83 & 2.20 & \bfseries 2.49 & 13,807.1 & 2,401.7 & 21,609.8 & 21,600.0 \\
8,000 & 12.57 & 1.76 & \bfseries 0.00 & 10.61 & 7.58 & \textemdash & \textemdash & \textemdash & 254.19 & 0.95 & \bfseries 4.88 & 16,983.6 & 4,306.1 & 21,616.7 & 21,600.0 \\
10,000 & 12.77 & 2.89 & \bfseries 0.00 & 9.36 & 6.60 & \textemdash & \textemdash & \textemdash & 319.08 & 2.54 & \bfseries 7.75 & 19,583.2 & 7,024.6 & 21,624.3 & 21,600.0 \\
 \midrule Avg &  &  & \bfseries 0.04 & 12.24 & 6.81 & \textemdash & \textemdash & \textemdash & 147.18 & 1.43 & \bfseries 2.76 & 9,591.2 & 2,485.3 & 13,915.5 & 19,586.7 \\
Min &  &  & \bfseries 0.00 & 8.58 & 4.94 & \textemdash & 3.57 & \textemdash & \bfseries 0.00 & 0.02 & \bfseries 0.04 & 108.7 & 48.4 & 107.6 & 9,508.0 \\
Max &  &  & \bfseries 0.20 & 17.54 & 9.86 & \textemdash & \textemdash & \textemdash & 319.08 & 2.54 & \bfseries 7.75 & 19,583.2 & 7,024.6 & 21,624.3 & 21,612.0 \\
\bottomrule \\ [-1.5ex] \multicolumn{8}{l}{\textbf{QKBP} is the contribution in this paper}
\end{tabularx}

	\caption{Results obtained with an extended time limit of 3,600 seconds for the instances of the \texttt{TeamFormation-QKP-2} collection: Each row represents a graph with a certain number of nodes $n$ and a density $\Delta$. For each graph, there are six instances with different budget values. These budget values are chosen as fractions of the total node weights of the graph. The fractions used are 0.025, 0.05, 0.1, 0.25, 0.5, and 0.75. The abbreviation OFV stands for objective function value and the column $t^{\textup{cut}}$ reports the running time of the simple parametric cut procedure in seconds. Note that the simple parametric cut procedure is applied only once for all six instances of the same graph. The hyphen (--) indicates that the respective approach did not find a solution for at least one of the six instances within the time limit.}\label{tbl:team_formation_qkp2_3600}
\end{table*}

The collection \textbf{TeamFormation-QKP-2} contains 36 instances with considerably higher densities compared to the instances from the collection TeamFormation-QKP-1. Table~\ref{tbl:team_formation_qkp2} presents the aggregated results for these instances. The proposed breakpoints approach considerably outperforms all other approaches in terms of objective value and running time. Due to the higher densities of the graphs, the Gurobi-based approach performs poorly on instances with $n\geq 2,000$. The Hexaly-based approach scales better than the Gurobi-based approach but its solutions can still deviate substantially from the best solutions in terms of objective function values. The greedy approach performs worse than the Hexaly-based approach on average. To analyze the impact of the time limit on the performance, we applied all approaches with a time limit of 3,600 seconds to the instances of the collection \textbf{TeamFormation-QKP-2}. 

\begin{table*}
	\small
	\begin{center}\setlength{\tabcolsep}{3pt}
\begin{tabular}{rr|rrrrrrrr|rrrrr}
 & \multicolumn{1}{c}{} & \multicolumn{13}{c}{\texttt{\texttt{TeamFormation-QKP-2} instance with $n=10{,}000$}} \\ \toprule
\multicolumn{2}{r}{} & \multicolumn{8}{c}{Deviation from best OFV (\%)} & \multicolumn{5}{c}{Running time (s)} \\ \cmidrule(rl){3-10} \cmidrule(rl){11-15}
 \multicolumn{1}{c}{$\gamma$} & Best OFV & \textbf{QKBP} & RG & IHEA & LDP & DP & QK & Gurobi & Hexaly & \textbf{QKBP} & RG & IHEA & Gurobi & Hexaly \\
\midrule
2.5 & 1,827.9 & \bfseries 0.00 & 33.38 & 29.30 & \textemdash & \textemdash & \textemdash & 1,058.22 & 10.15 & \bfseries 0.62 & 1,581.0 & 716.9 & 3,603.5 & 3,600.0 \\
5.0 & 3,518.4 & \bfseries 0.00 & 16.48 & 9.23 & \textemdash & \textemdash & \textemdash & 435.64 & 3.96 & \bfseries 0.45 & 3,600.1 & 877.6 & 3,603.5 & 3,600.0 \\
10.0 & 7,107.6 & \bfseries 0.00 & 5.22 & 1.01 & \textemdash & \textemdash & \textemdash & 245.85 & 0.46 & \bfseries 1.15 & 3,600.6 & 1,082.8 & 3,603.5 & 3,600.0 \\
25.0 & 18,590.7 & \bfseries 0.00 & 0.89 & 0.04 & \textemdash & \textemdash & \textemdash & 113.32 & 0.26 & \bfseries 1.80 & 3,600.5 & 1,448.8 & 3,606.7 & 3,600.0 \\
50.0 & 37,621.3 & \bfseries 0.00 & 0.19 & 0.04 & \textemdash & \textemdash & \textemdash & 43.85 & 0.30 & \bfseries 1.62 & 3,600.9 & 1,559.1 & 3,603.5 & 3,600.0 \\
75.0 & 54,050.1 & \bfseries 0.00 & 0.01 & 0.01 & \textemdash & \textemdash & \textemdash & 17.58 & 0.11 & \bfseries 2.10 & 3,600.1 & 1,339.4 & 3,603.6 & 3,600.0 \\
 \midrule Avg &  & \bfseries 0.00 & 9.36 & 6.60 & \textemdash & \textemdash & \textemdash & 319.08 & 2.54 & \bfseries 1.29 & 3,263.9 & 1,170.8 & 3,604.0 & 3,600.0 \\
Min &  & \bfseries 0.00 & 0.01 & 0.01 & \textemdash & \textemdash & \textemdash & 17.58 & 0.11 & \bfseries 0.45 & 1,581.0 & 716.9 & 3,603.5 & 3,600.0 \\
Max &  & \bfseries 0.00 & 33.38 & 29.30 & \textemdash & \textemdash & \textemdash & 1,058.22 & 10.15 & \bfseries 2.10 & 3,600.9 & 1,559.1 & 3,606.7 & 3,600.0 \\
\bottomrule \\ [-1.5ex] \multicolumn{8}{l}{\textbf{QKBP} is the contribution in this paper}
\end{tabular}
\end{center}
	\caption{Results obtained with an extended time limit of 3,600 seconds for the graph with $n=10{,}000$ of the \texttt{TeamFormation-QKP-2} collection: Each row reports the results for one budget value that is chosen as a fraction $\gamma$ of the total node weights of the graph. The abbreviation OFV stands for objective function value. The hyphen (--) indicates that the respective approach did not find a solution within the time limit.}\label{tbl:team_formation_qkp2_3600_instance_10000}
\end{table*}

As we can see from Table~\ref{tbl:team_formation_qkp2_3600}, the results of the approaches improve with the increased time limit, but the QKBP approach still provides the best solutions and is the fastest algorithm by a large margin. For the largest graph with $n=10{,}000$, Table~\ref{tbl:team_formation_qkp2_3600_instance_10000} lists the results for all six budget fractions $\gamma$. This table shows that the deviations from the best objective values vary a lot for the RG, IHEA, Gurobi-based, and Hexaly-based approaches and can be substantially higher than the average deviations reported in Table~\ref{tbl:team_formation_qkp2_3600}. For example, for $\gamma=2.5\%$ the deviations for the RG, IHEA, Gurobi-based, and Hexaly-based approaches are 33.38\%, 29.30\%, 1,058.22\%, and 10.15\%, respectively. Note that the largest deviations are not always obtained for the smallest budget fractions $\gamma=2.5\%$. Figure~\ref{fig:deviations_tf_2000} demonstrates this for the graph with $n=2{,}000$ from the collection TeamFormation-QKP-2.

Table~\ref{tbl:read_write_times} reports the running times associated with the simple parametric cut procedure in seconds when applied to the six graphs of the \textbf{TeamFormation-QKP-2} collection. The columns $t^{\textup{write}}$, $t^{\textup{cut}}$, $t^{\textup{read}}$ report the running times for writing the input file, executing the parametric cut procedure, and reading the output file, respectively.

\begin{table*}
	\small
	\begin{center}\setlength{\tabcolsep}{10pt}
\begin{tabular}{rrrrrr}
\toprule
$n$ & $\Delta$ & Number of edges & t$^{\textup{write}}$ & t$^{\textup{cut}}$ & t$^{\textup{read}}$ \\
\midrule
1,000 & 13.43 & 67,159 & 0.12 & 0.09 & 0.01 \\
2,000 & 12.38 & 247,696 & 0.45 & 0.22 & 0.01 \\
4,000 & 12.68 & 1,014,045 & 1.64 & 0.56 & 0.02 \\
6,000 & 12.54 & 2,257,990 & 3.86 & 1.08 & 0.03 \\
8,000 & 12.57 & 4,023,218 & 7.19 & 1.76 & 0.03 \\
10,000 & 12.77 & 6,383,021 & 11.18 & 2.89 & 0.04 \\
\bottomrule
\end{tabular}
\end{center}
	\caption{Running times associated with simple parametric cut procedure in seconds when applied to the six graphs of the \texttt{TeamFormation-QKP-2} collection. The columns $t^{\textup{write}}$, $t^{\textup{cut}}$, $t^{\textup{read}}$ report the running times for writing the input file, executing the parametric cut procedure, and reading the output file, respectively.}\label{tbl:read_write_times}
\end{table*}

\subsection{Discussion of results}\label{sec:results_discussion}

Table~\ref{tbl:all_collections} provides a summary of the results for all dataset collections. For each collection and each approach, the table reports the average deviation from the best objective function values and the sum of the running times for all the respective instances. 

\begin{table*}
	\scriptsize
	\setlength{\tabcolsep}{4pt}
\begin{tabular}{l|rrrrrrrr|rrrrrrrr}
\toprule
 \multicolumn{1}{l}{} & \multicolumn{8}{c}{Average deviation from best OFV (\%)} & \multicolumn{8}{c}{Sum of running time (s)} \\ \cmidrule(rl){2-9} \cmidrule(rl){10-17}\
Collection & \textbf{QKBP} & RG & IHEA & LDP & DP & QK & Gurobi & Hexaly & \textbf{QKBP} & RG & IHEA & LDP & DP & QK & Gurobi & Hexaly \\
\midrule
Standard-QKP & 0.43 & 0.19 & \bfseries 0.00 & \textemdash & 0.04 & \textemdash & 0.02 & 0.14 & \bfseries 0.5 & 56.7 & 234.3 & 7,621.8 & 354.5 & 1,193.9 & 3,749.5 & 7,673.0 \\
QKP-GroupII & 0.12 & 0.05 & \bfseries 0.00 & \textemdash & \textemdash & \textemdash & 526.94 & 4.46 & \bfseries 13.9 & 5,621.5 & 2,834.6 & 9,600.0 & 9,600.0 & 9,600.0 & 9,658.0 & 9,645.0 \\
QKP-GroupIII & 0.03 & 0.03 & \bfseries 0.00 & \textemdash & \textemdash & \textemdash & \textemdash & 19.30 & \bfseries 100.4 & 4,821.3 & 6,462.0 & 4,800.0 & 4,800.0 & 4,800.0 & 4,867.8 & 4,805.0 \\
Large-QKP & 0.14 & 0.05 & \bfseries 0.00 & \textemdash & \textemdash & \textemdash & \textemdash & 2.04 & \bfseries 32.3 & 7,696.0 & 6,873.0 & 17,280.0 & 13,271.8 & 17,280.0 & 14,813.2 & 16,758.0 \\
Dispersion-QKP & 0.23 & 0.11 & \bfseries 0.00 & \textemdash & \textemdash & \textemdash & 388.17 & 1.24 & \bfseries 15.0 & 18,491.3 & 9,905.3 & 69,120.0 & 41,911.0 & 66,759.6 & 47,586.3 & 62,442.0 \\
TF-QKP-1 & 0.04 & 48.64 & 38.26 & \textemdash & \textemdash & \textemdash & \bfseries 0.00 & 0.73 & \bfseries 49.0 & 9,496.5 & 12,354.6 & 10,080.0 & 10,080.0 & 10,080.0 & 533.8 & 5,102.0 \\
TF-QKP-2 & \bfseries 0.02 & 12.36 & 6.95 & \textemdash & \textemdash & \textemdash & \textemdash & 7.38 & \bfseries 22.5 & 3,459.2 & 4,190.8 & 4,320.0 & 4,320.0 & 4,320.0 & 3,744.4 & 4,332.0 \\
 \midrule Avg & \bfseries 0.14 & 8.78 & 6.46 & \textemdash & \textemdash & \textemdash & \textemdash & 5.04 & \bfseries 33.4 & 7,091.8 & 6,122.1 & 17,546.0 & 12,048.2 & 16,290.5 & 12,136.1 & 15,822.4 \\
Min & 0.02 & 0.03 & \bfseries 0.00 & \textemdash & 0.04 & \textemdash & \bfseries 0.00 & 0.14 & \bfseries 0.5 & 56.7 & 234.3 & 4,320.0 & 354.5 & 1,193.9 & 533.8 & 4,332.0 \\
Max & \bfseries 0.43 & 48.64 & 38.26 & \textemdash & \textemdash & \textemdash & \textemdash & 19.30 & \bfseries 100.4 & 18,491.3 & 12,354.6 & 69,120.0 & 41,911.0 & 66,759.6 & 47,586.3 & 62,442.0 \\
\bottomrule \multicolumn{8}{l}{} \\ [-1.75ex] \multicolumn{8}{l}{\textbf{QKBP} is the contribution in this paper}
\end{tabular}

	\caption{Summary of results for all collections: For each collection and algorithm, the table reports the average deviation from the best objective function values and the sum of running times of all the respective instances. The hyphen (--) indicates that the respective algorithm did not find a solution for at least one of the respective instances, within the time limit of 120 seconds.}\label{tbl:all_collections}
\end{table*}

We assess the results as follows: The QKBP approach is the fastest and most robust approach, consistently delivering high-quality solutions across all collections. In contrast, all other approaches struggle with some of the collections. IHEA typically generates the best solutions, especially for large instances, but along with RG, it is not competitive on very sparse instances like those in the team formation collections. The good performance of RG on most instances, however, is notable. Among the exact approaches (QK, Gurobi, Hexaly), only Hexaly scales well to large instances, though it is outperformed by the best heuristics in terms of solution quality and speed. The tested dynamic programming-based approaches (LDP, DP, QK) are not competitive on medium and large instances. Overall, QKBP excels across all instance types, densities, and budgets, achieving the best or near-best solutions in a fraction of the time required by other algorithms.

\section{Conclusions}\label{sec:conclusions}
We introduce here the use of the breakpoints algorithm (QKBP) for the Quadratic Knapsack Problem. The algorithm uses the breakpoints in the concave envelope of solutions to the parametric relaxation of the problem, and determines the interval between two consecutive breakpoints that contain the budget.  The algorithm makes use of the fact that each breakpoint corresponds to an optimal solution and the solutions for consecutive breakpoints are nested, and applies a simple greedy procedure to append or remove items from the breakpoint solution in order to attain a feasible solution for the given budget.
The breakpoints are attained efficiently by a parametric cut procedure, on a compact formulation of the problem, generated by a general purpose procedure. 

The paper provides an extensive experimental study that includes many of the known datasets as well as new datasets and large scale instances not previously studied.   We compare the performance of QKBP to that of leading state-of-the-art algorithms including DP, LDP, QK, RG and IHEA as well as integer programming solvers Gurobi and Hexaly (see Table \ref{table_tested_algorithms} for description and codes).
While the performance of all these algorithms vary, depending on the size of the problem, or density, or budget fraction, QKBP is consistently delivering optimal or close to optimal solutions, with running times that are orders of magnitude faster that those of the other algorithms.

In future research, we plan to compare our approach with newly developed heuristics for the QKP, such as the one recently presented by \cite{fennich2024novel}.

\noindent
\section*{Acknowledgement}
The first author was supported in part by AI institute NSF award 2112533. We would like to express our gratitude to Dr.~Yuning Chen and Prof.~Dr.~Jin-Kao Hao for providing the IHEA algorithm's source code and for their assistance with its installation. We are also grateful to Prof.~Dr.~Djeumou Fomeni for providing the C source code of the LDP approach.  

\bibliography{references.bib}

\appendix

\section{Additional results for instances of the Dispersion-QKP collection}

\begin{table*}
	\scriptsize
	\setlength{\tabcolsep}{4.7pt}
\begin{tabularx}{\textwidth}{rrr|rrrrrrrr|rrrrrrrr}
 \multicolumn{19}{c}{\texttt{Dispersion-QKP (wgeo)} collection} \\ \toprule
\multicolumn{3}{r}{} & \multicolumn{8}{c}{Avg deviation from best OFV (\%)} & \multicolumn{8}{c}{Sum runtime over six budgets (s)} \\ \cmidrule(rl){4-11} \cmidrule(rl){12-19}\
$n$ & $\Delta$ & t$^{\textup{cut}}$ & \textbf{QKBP} & RG & IHEA & LDP & DP & QK & Gurobi & Hexaly & \textbf{QKBP} & RG & IHEA & LDP & DP & QK & Gurobi & Hexaly \\
\midrule
300 & 5 & 0.02 & 0.28 & 0.54 & 0.07 & \textemdash & 414.56 & \textemdash & \bfseries 0.00 & \bfseries 0.00 & \bfseries 0.02 & 5.7 & 21.7 & 720.0 & 45.9 & 720.0 & 4.0 & 82.0 \\
300 & 10 & 0.02 & 0.17 & 0.10 & \bfseries 0.00 & \textemdash & 1,104.19 & \textemdash & \bfseries 0.00 & \bfseries 0.00 & \bfseries 0.02 & 5.8 & 21.6 & 720.0 & 39.5 & 720.0 & 9.8 & 386.0 \\
300 & 25 & 0.02 & 0.25 & 0.07 & 0.01 & \textemdash & 911.35 & \textemdash & \bfseries 0.00 & 0.01 & \bfseries 0.02 & 6.1 & 26.9 & 720.0 & 39.6 & 720.0 & 54.8 & 518.0 \\
300 & 50 & 0.02 & 0.44 & 0.03 & \bfseries 0.00 & \textemdash & 679.77 & \textemdash & \bfseries 0.00 & 0.02 & \bfseries 0.02 & 6.0 & 25.3 & 720.0 & 39.0 & 720.0 & 288.2 & 720.0 \\
300 & 75 & 0.02 & 0.32 & 0.11 & \bfseries 0.00 & \textemdash & 431.43 & \textemdash & \bfseries 0.00 & 0.01 & \bfseries 0.02 & 6.1 & 24.1 & 720.0 & 39.7 & 720.0 & 407.4 & 720.0 \\
300 & 100 & 0.03 & 0.46 & 0.09 & \bfseries 0.00 & \textemdash & 532.23 & \textemdash & \bfseries 0.00 & 0.01 & \bfseries 0.02 & 6.2 & 26.7 & 720.0 & 36.3 & 720.0 & 260.7 & 720.0 \\
500 & 5 & 0.03 & 0.46 & 0.32 & 0.03 & \textemdash & 296.41 & \textemdash & \bfseries 0.00 & 0.01 & \bfseries 0.04 & 19.2 & 32.5 & 720.0 & 294.2 & 720.0 & 13.7 & 394.0 \\
500 & 10 & 0.03 & 0.13 & 0.15 & \bfseries 0.00 & \textemdash & 945.80 & \textemdash & \bfseries 0.00 & 0.03 & \bfseries 0.03 & 19.5 & 31.0 & 720.0 & 223.0 & 720.0 & 110.5 & 632.0 \\
500 & 25 & 0.03 & 0.29 & 0.11 & \bfseries 0.00 & \textemdash & 599.04 & \textemdash & \bfseries 0.00 & 0.13 & \bfseries 0.04 & 19.9 & 43.4 & 720.0 & 228.9 & 720.0 & 218.7 & 720.0 \\
500 & 50 & 0.03 & 0.24 & 0.09 & \bfseries 0.00 & \textemdash & 453.14 & \textemdash & 1.33 & 0.07 & \bfseries 0.04 & 20.3 & 36.2 & 720.0 & 223.5 & 720.0 & 655.3 & 720.0 \\
500 & 75 & 0.03 & 0.22 & 0.02 & \bfseries 0.00 & \textemdash & 541.27 & \textemdash & 678.14 & 0.09 & \bfseries 0.04 & 20.9 & 32.8 & 720.0 & 229.3 & 720.0 & 720.5 & 720.0 \\
500 & 100 & 0.03 & 0.13 & 0.02 & \bfseries 0.00 & \textemdash & 689.25 & \textemdash & 758.77 & 0.07 & \bfseries 0.03 & 19.7 & 40.4 & 720.0 & 219.8 & 720.0 & 720.5 & 720.0 \\
1,000 & 5 & 0.09 & 0.13 & 0.06 & 0.01 & \textemdash & \textemdash & \textemdash & \bfseries 0.00 & 0.17 & \bfseries 0.09 & 132.1 & 69.6 & 720.0 & 720.0 & 720.0 & 186.8 & 720.0 \\
1,000 & 10 & 0.11 & 0.05 & 0.04 & \bfseries 0.00 & \textemdash & \textemdash & \textemdash & \bfseries 0.00 & 0.10 & \bfseries 0.10 & 135.3 & 68.8 & 720.0 & 720.0 & 720.0 & 368.7 & 720.0 \\
1,000 & 25 & 0.09 & 0.08 & 0.01 & \bfseries 0.00 & \textemdash & \textemdash & \textemdash & 754.81 & 0.13 & \bfseries 0.10 & 132.1 & 58.3 & 720.0 & 720.0 & 720.0 & 636.7 & 720.0 \\
1,000 & 50 & 0.11 & 0.13 & \bfseries 0.00 & \bfseries 0.00 & \textemdash & \textemdash & \textemdash & 1,007.67 & 0.30 & \bfseries 0.10 & 141.6 & 74.0 & 720.0 & 720.0 & 720.0 & 665.9 & 720.0 \\
1,000 & 75 & 0.14 & 0.09 & 0.03 & \bfseries 0.00 & \textemdash & \textemdash & \textemdash & 863.49 & 0.55 & \bfseries 0.10 & 134.8 & 87.9 & 720.0 & 720.0 & 720.0 & 722.3 & 720.0 \\
1,000 & 100 & 0.12 & 0.15 & \bfseries 0.00 & \bfseries 0.00 & \textemdash & \textemdash & \textemdash & 990.27 & 0.70 & \bfseries 0.10 & 133.1 & 84.2 & 720.0 & 720.0 & 720.0 & 722.8 & 723.0 \\
2,000 & 5 & 0.17 & 0.03 & 0.03 & \bfseries 0.00 & \textemdash & \textemdash & \textemdash & 566.22 & 0.18 & \bfseries 0.35 & 591.8 & 249.4 & 720.0 & 720.0 & 720.0 & 528.4 & 720.0 \\
2,000 & 10 & 0.20 & 0.03 & \bfseries 0.00 & \bfseries 0.00 & \textemdash & \textemdash & \textemdash & 783.17 & 0.30 & \bfseries 0.42 & 599.5 & 268.1 & 720.0 & 720.0 & 720.0 & 706.4 & 720.0 \\
2,000 & 25 & 0.25 & 0.06 & 0.01 & \bfseries 0.00 & \textemdash & \textemdash & \textemdash & 8.49 & 0.84 & \bfseries 0.43 & 603.0 & 217.2 & 720.0 & 720.0 & 720.0 & 722.4 & 723.0 \\
2,000 & 50 & 0.33 & 0.04 & 0.01 & \bfseries 0.00 & \textemdash & \textemdash & \textemdash & 1,035.01 & 1.71 & \bfseries 0.44 & 605.6 & 254.3 & 720.0 & 720.0 & 720.0 & 733.4 & 732.0 \\
2,000 & 75 & 0.42 & 0.07 & \bfseries 0.00 & \bfseries 0.00 & \textemdash & \textemdash & \textemdash & 1,034.69 & 2.57 & \bfseries 0.43 & 604.6 & 264.7 & 720.0 & 720.0 & 720.0 & 733.1 & 720.0 \\
2,000 & 100 & 0.56 & 0.05 & \bfseries 0.00 & \bfseries 0.00 & \textemdash & \textemdash & \textemdash & 1,024.97 & 2.87 & \bfseries 0.42 & 605.5 & 246.3 & 720.0 & 720.0 & 720.0 & 730.8 & 720.0 \\
 \midrule Avg &  &  & 0.18 & 0.08 & \bfseries 0.01 & \textemdash & \textemdash & \textemdash & 396.13 & 0.45 & \bfseries 0.14 & 190.6 & 96.0 & 720.0 & 429.1 & 720.0 & 455.1 & 654.6 \\
Min &  &  & 0.03 & \bfseries 0.00 & \bfseries 0.00 & \textemdash & 296.41 & \textemdash & \bfseries 0.00 & \bfseries 0.00 & \bfseries 0.02 & 5.7 & 21.6 & 720.0 & 36.3 & 720.0 & 4.0 & 82.0 \\
Max &  &  & 0.46 & 0.54 & \bfseries 0.07 & \textemdash & \textemdash & \textemdash & 1,035.01 & 2.87 & \bfseries 0.44 & 605.6 & 268.1 & 720.0 & 720.0 & 720.0 & 733.4 & 732.0 \\
\bottomrule \\ [-1.5ex] \multicolumn{8}{l}{\textbf{QKBP} is the contribution in this paper}
\end{tabularx}

	\caption{Results for the instances of the \texttt{Dispersion-QKP} collection with distance computation strategy \texttt{wgeo}: Each row represents a graph with a certain number of nodes $n$ and a density $\Delta$. For each graph, there are six instances with different budget values. These budget values are chosen as fractions of the total node weights of the graph. The fractions used are 0.025, 0.05, 0.1, 0.25, 0.5, 0.75. The abbreviation OFV stands for objective function value and the column $t^{\textup{cut}}$ reports the running time of the simple parametric cut procedure in seconds. Note that the simple parametric cut procedure is applied only once for all six instances of the same graph. The time limit for each instance is 120 seconds. The hyphen (--) indicates that the respective approach did not find a solution for at least one of the six instances within the time limit.}\label{tbl:dispersion_qkp_wgeo}
\end{table*}

\begin{table*}
	\scriptsize
	\setlength{\tabcolsep}{4.8pt}
\begin{tabularx}{\textwidth}{rrr|rrrrrrrr|rrrrrrrr}
 \multicolumn{19}{c}{\texttt{Dispersion-QKP (expo)} collection} \\ \toprule
\multicolumn{3}{r}{} & \multicolumn{8}{c}{Avg deviation from best OFV (\%)} & \multicolumn{8}{c}{Sum runtime over six budgets (s)} \\ \cmidrule(rl){4-11} \cmidrule(rl){12-19}\
$n$ & $\Delta$ & t$^{\textup{cut}}$ & \textbf{QKBP} & RG & IHEA & LDP & DP & QK & Gurobi & Hexaly & \textbf{QKBP} & RG & IHEA & LDP & DP & QK & Gurobi & Hexaly \\
\midrule
300 & 5 & 0.02 & 0.80 & 0.86 & 0.10 & \textemdash & 46.87 & \textemdash & \bfseries 0.00 & \bfseries 0.00 & \bfseries 0.02 & 5.8 & 21.0 & 720.0 & 51.0 & 720.0 & 4.5 & 103.0 \\
300 & 10 & 0.03 & 0.38 & 0.30 & \bfseries 0.00 & \textemdash & 136.67 & \textemdash & \bfseries 0.00 & \bfseries 0.00 & \bfseries 0.02 & 6.1 & 23.8 & 720.0 & 46.2 & 365.6 & 8.8 & 354.0 \\
300 & 25 & 0.02 & 0.24 & 0.06 & \bfseries 0.00 & \textemdash & 161.46 & \textemdash & \bfseries 0.00 & \bfseries 0.00 & \bfseries 0.02 & 6.4 & 22.1 & 720.0 & 43.3 & 720.0 & 43.3 & 506.0 \\
300 & 50 & 0.03 & 0.55 & 0.19 & \bfseries 0.00 & \textemdash & 230.02 & \textemdash & \bfseries 0.00 & 0.43 & \bfseries 0.02 & 6.1 & 24.3 & 720.0 & 42.0 & 720.0 & 284.4 & 720.0 \\
300 & 75 & 0.02 & 0.90 & 0.40 & \bfseries 0.00 & \textemdash & 248.46 & \textemdash & \bfseries 0.00 & 0.66 & \bfseries 0.02 & 6.3 & 28.3 & 720.0 & 42.3 & 720.0 & 522.6 & 720.0 \\
300 & 100 & 0.02 & 0.65 & 0.48 & \bfseries 0.00 & \textemdash & 270.55 & \textemdash & 0.04 & 1.09 & \bfseries 0.02 & 6.2 & 26.3 & 720.0 & 41.6 & 720.0 & 619.7 & 720.0 \\
500 & 5 & 0.03 & 0.32 & 0.48 & 0.02 & \textemdash & 65.98 & \textemdash & \bfseries 0.00 & \bfseries 0.00 & \bfseries 0.03 & 19.4 & 35.0 & 720.0 & 335.5 & 720.0 & 11.9 & 450.0 \\
500 & 10 & 0.03 & 0.28 & 0.08 & \bfseries 0.00 & \textemdash & 139.58 & \textemdash & \bfseries 0.00 & 0.03 & \bfseries 0.04 & 19.9 & 36.5 & 720.0 & 281.6 & 720.0 & 177.4 & 609.0 \\
500 & 25 & 0.03 & 0.12 & 0.05 & \bfseries 0.00 & \textemdash & 245.57 & \textemdash & 0.01 & 0.37 & \bfseries 0.04 & 20.6 & 37.6 & 720.0 & 256.5 & 720.0 & 672.2 & 720.0 \\
500 & 50 & 0.03 & 0.34 & 0.13 & \bfseries 0.00 & \textemdash & 280.36 & \textemdash & 1.32 & 0.97 & \bfseries 0.04 & 21.2 & 37.3 & 720.0 & 252.7 & 720.0 & 720.6 & 720.0 \\
500 & 75 & 0.03 & 0.30 & 0.09 & \bfseries 0.00 & \textemdash & 330.39 & \textemdash & 556.88 & 1.35 & \bfseries 0.04 & 21.1 & 40.7 & 720.0 & 244.1 & 720.0 & 720.5 & 720.0 \\
500 & 100 & 0.03 & 0.48 & 0.16 & \bfseries 0.00 & \textemdash & 351.33 & \textemdash & 816.88 & 1.98 & \bfseries 0.04 & 21.1 & 37.4 & 720.0 & 247.8 & 720.0 & 720.5 & 720.0 \\
1,000 & 5 & 0.09 & 0.15 & 0.11 & 0.01 & \textemdash & \textemdash & \textemdash & \bfseries 0.00 & 0.16 & \bfseries 0.09 & 132.5 & 64.2 & 720.0 & 720.0 & 720.0 & 97.2 & 720.0 \\
1,000 & 10 & 0.08 & 0.15 & 0.07 & \bfseries 0.00 & \textemdash & \textemdash & \textemdash & 0.41 & 0.33 & \bfseries 0.10 & 134.4 & 61.5 & 720.0 & 720.0 & 720.0 & 721.9 & 720.0 \\
1,000 & 25 & 0.08 & 0.15 & 0.02 & \bfseries 0.00 & \textemdash & \textemdash & \textemdash & 0.60 & 0.80 & \bfseries 0.11 & 136.9 & 77.0 & 720.0 & 720.0 & 720.0 & 647.0 & 720.0 \\
1,000 & 50 & 0.12 & 0.11 & 0.03 & \bfseries 0.00 & \textemdash & \textemdash & \textemdash & 625.71 & 2.18 & \bfseries 0.11 & 137.2 & 89.8 & 720.0 & 720.0 & 720.0 & 722.1 & 720.0 \\
1,000 & 75 & 0.12 & 0.05 & 0.02 & \bfseries 0.00 & \textemdash & \textemdash & \textemdash & 782.95 & 3.15 & \bfseries 0.11 & 137.5 & 91.4 & 720.0 & 720.0 & 720.0 & 722.3 & 720.0 \\
1,000 & 100 & 0.16 & 0.12 & 0.02 & \bfseries 0.00 & \textemdash & \textemdash & \textemdash & 798.94 & 3.74 & \bfseries 0.11 & 137.6 & 97.2 & 720.0 & 720.0 & 720.0 & 728.9 & 723.0 \\
2,000 & 5 & 0.17 & 0.04 & 0.02 & \bfseries 0.00 & \textemdash & \textemdash & \textemdash & 2.82 & 0.40 & \bfseries 0.38 & 600.4 & 250.8 & 720.0 & 720.0 & 720.0 & 649.2 & 720.0 \\
2,000 & 10 & 0.22 & 0.06 & 0.01 & \bfseries 0.00 & \textemdash & \textemdash & \textemdash & 899.22 & 0.90 & \bfseries 0.41 & 605.3 & 246.8 & 720.0 & 720.0 & 720.0 & 724.9 & 720.0 \\
2,000 & 25 & 0.25 & 0.07 & 0.02 & \bfseries 0.00 & \textemdash & \textemdash & \textemdash & 889.19 & 2.94 & \bfseries 0.49 & 611.4 & 301.3 & 720.0 & 720.0 & 720.0 & 740.9 & 721.0 \\
2,000 & 50 & 0.34 & 0.11 & 0.03 & \bfseries 0.00 & \textemdash & \textemdash & \textemdash & 1,035.36 & 4.85 & \bfseries 0.57 & 612.2 & 312.9 & 720.0 & 720.0 & 720.0 & 731.8 & 733.0 \\
2,000 & 75 & 0.47 & 0.07 & 0.04 & \bfseries 0.00 & \textemdash & \textemdash & \textemdash & 914.84 & 6.52 & \bfseries 0.49 & 612.8 & 305.3 & 720.0 & 720.0 & 720.0 & 732.8 & 720.0 \\
2,000 & 100 & 0.58 & 0.10 & 0.07 & \bfseries 0.00 & \textemdash & \textemdash & \textemdash & 1,005.41 & 7.51 & \bfseries 0.49 & 614.4 & 300.0 & 720.0 & 720.0 & 720.0 & 732.9 & 720.0 \\
 \midrule Avg &  &  & 0.27 & 0.16 & \bfseries 0.01 & \textemdash & \textemdash & \textemdash & 347.11 & 1.68 & \bfseries 0.16 & 193.0 & 107.0 & 720.0 & 438.5 & 705.2 & 519.1 & 655.0 \\
Min &  &  & 0.04 & 0.01 & \bfseries 0.00 & \textemdash & 46.87 & \textemdash & \bfseries 0.00 & \bfseries 0.00 & \bfseries 0.02 & 5.8 & 21.0 & 720.0 & 41.6 & 365.6 & 4.5 & 103.0 \\
Max &  &  & 0.90 & 0.86 & \bfseries 0.10 & \textemdash & \textemdash & \textemdash & 1,035.36 & 7.51 & \bfseries 0.57 & 614.4 & 312.9 & 720.0 & 720.0 & 720.0 & 740.9 & 733.0 \\
\bottomrule \\ [-1.5ex] \multicolumn{8}{l}{\textbf{QKBP} is the contribution in this paper}
\end{tabularx}

	\caption{Results for the instances of the \texttt{Dispersion-QKP} collection with distance computation strategy \texttt{expo}: Each row represents a graph with a certain number of nodes $n$ and a density $\Delta$. For each graph, there are six instances with different budget values. These budget values are chosen as fractions of the total node weights of the graph. The fractions used are 0.025, 0.05, 0.1, 0.25, 0.5, 0.75. The abbreviation OFV stands for objective function value and the column $t^{\textup{cut}}$ reports the running time of the simple parametric cut procedure in seconds. Note that the simple parametric cut procedure is applied only once for all six instances of the same graph. The time limit for each instance is 120 seconds. The hyphen (--) indicates that the respective approach did not find a solution for at least one of the six instances within the time limit.}\label{tbl:dispersion_qkp_expo}
\end{table*}

\begin{table*}
	\scriptsize
	\setlength{\tabcolsep}{4.8pt}
\begin{tabularx}{\textwidth}{rrr|rrrrrrrr|rrrrrrrr}
 \multicolumn{19}{c}{\texttt{Dispersion-QKP (ran)} collection} \\ \toprule
\multicolumn{3}{r}{} & \multicolumn{8}{c}{Avg deviation from best OFV (\%)} & \multicolumn{8}{c}{Sum runtime over six budgets (s)} \\ \cmidrule(rl){4-11} \cmidrule(rl){12-19}\
$n$ & $\Delta$ & t$^{\textup{cut}}$ & \textbf{QKBP} & RG & IHEA & LDP & DP & QK & Gurobi & Hexaly & \textbf{QKBP} & RG & IHEA & LDP & DP & QK & Gurobi & Hexaly \\
\midrule
300 & 5 & 0.02 & 0.19 & 0.68 & \bfseries 0.00 & \textemdash & 40.96 & 5.45 & \bfseries 0.00 & \bfseries 0.00 & \bfseries 0.02 & 5.8 & 23.9 & 720.0 & 53.4 & 8.8 & 2.4 & 62.0 \\
300 & 10 & 0.02 & 0.57 & 0.06 & \bfseries 0.00 & \textemdash & 69.61 & \textemdash & \bfseries 0.00 & \bfseries 0.00 & \bfseries 0.02 & 5.9 & 24.3 & 720.0 & 48.2 & 720.0 & 12.6 & 265.0 \\
300 & 25 & 0.02 & 0.30 & 0.05 & \bfseries 0.00 & \textemdash & 200.44 & \textemdash & \bfseries 0.00 & 0.07 & \bfseries 0.02 & 6.3 & 21.1 & 720.0 & 47.6 & 601.5 & 127.2 & 540.0 \\
300 & 50 & 0.03 & 0.46 & 0.19 & \bfseries 0.00 & \textemdash & 219.04 & \textemdash & \bfseries 0.00 & 0.57 & \bfseries 0.02 & 6.5 & 27.0 & 720.0 & 45.2 & 720.0 & 515.1 & 654.0 \\
300 & 75 & 0.02 & 0.47 & 0.22 & \bfseries 0.00 & \textemdash & 260.14 & \textemdash & \bfseries 0.00 & 1.29 & \bfseries 0.02 & 6.3 & 22.5 & 720.0 & 44.2 & 720.0 & 546.7 & 677.0 \\
300 & 100 & 0.02 & 0.48 & 0.38 & \bfseries 0.00 & \textemdash & 303.33 & \textemdash & 0.07 & 1.03 & \bfseries 0.02 & 6.3 & 32.3 & 720.0 & 46.8 & 369.1 & 497.9 & 687.0 \\
500 & 5 & 0.03 & 0.54 & 0.30 & 0.01 & \textemdash & 63.23 & \textemdash & \bfseries 0.00 & \bfseries 0.00 & \bfseries 0.04 & 19.9 & 38.1 & 720.0 & 326.2 & 720.0 & 30.5 & 317.0 \\
500 & 10 & 0.03 & 0.27 & 0.14 & \bfseries 0.00 & \textemdash & 102.12 & \textemdash & \bfseries 0.00 & 0.03 & \bfseries 0.04 & 20.2 & 30.6 & 720.0 & 304.8 & 720.0 & 71.2 & 654.0 \\
500 & 25 & 0.03 & 0.34 & 0.07 & \bfseries 0.00 & \textemdash & 179.05 & \textemdash & 0.03 & 0.52 & \bfseries 0.04 & 21.0 & 39.7 & 720.0 & 288.5 & 720.0 & 645.3 & 720.0 \\
500 & 50 & 0.03 & 0.24 & 0.11 & \bfseries 0.00 & \textemdash & 285.57 & \textemdash & 0.08 & 1.42 & \bfseries 0.04 & 20.8 & 36.4 & 720.0 & 268.4 & 720.0 & 721.3 & 720.0 \\
500 & 75 & 0.03 & 0.33 & 0.18 & \bfseries 0.00 & \textemdash & 408.38 & \textemdash & 609.95 & 2.24 & \bfseries 0.04 & 21.2 & 36.2 & 720.0 & 255.3 & 720.0 & 720.4 & 720.0 \\
500 & 100 & 0.05 & 0.30 & 0.18 & \bfseries 0.00 & \textemdash & 404.08 & \textemdash & 835.70 & 3.26 & \bfseries 0.04 & 21.6 & 36.5 & 720.0 & 263.5 & 720.0 & 720.4 & 720.0 \\
1,000 & 5 & 0.08 & 0.10 & 0.06 & \bfseries 0.00 & \textemdash & \textemdash & \textemdash & \bfseries 0.00 & 0.14 & \bfseries 0.10 & 134.6 & 71.5 & 720.0 & 720.0 & 720.0 & 224.3 & 720.0 \\
1,000 & 10 & 0.08 & 0.07 & 0.04 & \bfseries 0.00 & \textemdash & \textemdash & \textemdash & \bfseries 0.00 & 0.49 & \bfseries 0.11 & 140.3 & 76.4 & 720.0 & 720.0 & 720.0 & 507.3 & 720.0 \\
1,000 & 25 & 0.08 & 0.10 & 0.04 & \bfseries 0.00 & \textemdash & \textemdash & \textemdash & 1,003.14 & 1.20 & \bfseries 0.11 & 136.3 & 86.7 & 720.0 & 720.0 & 720.0 & 720.5 & 720.0 \\
1,000 & 50 & 0.11 & 0.22 & 0.09 & \bfseries 0.00 & \textemdash & \textemdash & \textemdash & 614.70 & 2.57 & \bfseries 0.11 & 142.2 & 94.9 & 720.0 & 720.0 & 720.0 & 722.0 & 720.0 \\
1,000 & 75 & 0.14 & 0.37 & 0.12 & \bfseries 0.00 & \textemdash & \textemdash & \textemdash & 823.47 & 3.38 & \bfseries 0.11 & 139.8 & 94.0 & 720.0 & 720.0 & 720.0 & 722.3 & 720.0 \\
1,000 & 100 & 0.14 & 0.07 & 0.05 & \bfseries 0.00 & \textemdash & \textemdash & \textemdash & 768.95 & 4.33 & \bfseries 0.11 & 140.5 & 96.6 & 720.0 & 720.0 & 720.0 & 731.9 & 722.0 \\
2,000 & 5 & 0.19 & 0.05 & 0.03 & \bfseries 0.00 & \textemdash & \textemdash & \textemdash & 386.77 & 0.53 & \bfseries 0.40 & 603.5 & 244.2 & 720.0 & 720.0 & 720.0 & 720.6 & 720.0 \\
2,000 & 10 & 0.20 & 0.08 & \bfseries 0.00 & \bfseries 0.00 & \textemdash & \textemdash & \textemdash & 26.91 & 1.08 & \bfseries 0.48 & 604.7 & 266.3 & 720.0 & 720.0 & 720.0 & 709.8 & 720.0 \\
2,000 & 25 & 0.26 & 0.12 & 0.03 & \bfseries 0.00 & \textemdash & \textemdash & \textemdash & 1,086.93 & 2.98 & \bfseries 0.49 & 610.7 & 294.8 & 720.0 & 720.0 & 720.0 & 729.2 & 723.0 \\
2,000 & 50 & 0.36 & 0.08 & 0.04 & \bfseries 0.00 & \textemdash & \textemdash & \textemdash & 1,071.28 & 4.85 & \bfseries 0.49 & 611.8 & 295.3 & 720.0 & 720.0 & 720.0 & 726.5 & 732.0 \\
2,000 & 75 & 0.47 & 0.10 & 0.07 & \bfseries 0.00 & \textemdash & \textemdash & \textemdash & 972.89 & 6.45 & \bfseries 0.49 & 611.5 & 293.6 & 720.0 & 720.0 & 720.0 & 729.0 & 720.0 \\
2,000 & 100 & 0.55 & 0.08 & 0.07 & \bfseries 0.00 & \textemdash & \textemdash & \textemdash & 1,006.27 & 8.46 & \bfseries 0.49 & 615.7 & 332.4 & 720.0 & 720.0 & 720.0 & 730.7 & 720.0 \\
 \midrule Avg &  &  & 0.25 & 0.13 & \bfseries 0.00 & \textemdash & \textemdash & \textemdash & 383.63 & 1.95 & \bfseries 0.16 & 193.9 & 109.0 & 720.0 & 443.0 & 670.8 & 524.4 & 641.4 \\
Min &  &  & 0.05 & \bfseries 0.00 & \bfseries 0.00 & \textemdash & 40.96 & 5.45 & \bfseries 0.00 & \bfseries 0.00 & \bfseries 0.02 & 5.8 & 21.1 & 720.0 & 44.2 & 8.8 & 2.4 & 62.0 \\
Max &  &  & 0.57 & 0.68 & \bfseries 0.01 & \textemdash & \textemdash & \textemdash & 1,086.93 & 8.46 & \bfseries 0.49 & 615.7 & 332.4 & 720.0 & 720.0 & 720.0 & 731.9 & 732.0 \\
\bottomrule \\ [-1.5ex] \multicolumn{8}{l}{\textbf{QKBP} is the contribution in this paper}
\end{tabularx}

	\caption{Results for the instances of the \texttt{Dispersion-QKP} collection with distance computation strategy \texttt{ran}: Each row represents a graph with a certain number of nodes $n$ and a density $\Delta$. For each graph, there are six instances with different budget values. These budget values are chosen as fractions of the total node weights of the graph. The fractions used are 0.025, 0.05, 0.1, 0.25, 0.5, 0.75. The abbreviation OFV stands for objective function value and the column $t^{\textup{cut}}$ reports the running time of the simple parametric cut procedure in seconds. Note that the simple parametric cut procedure is applied only once for all six instances of the same graph. The time limit for each instance is 120 seconds. The hyphen (--) indicates that the respective approach did not find a solution for at least one of the six instances within the time limit.}\label{tbl:dispersion_qkp_ran}
\end{table*}

\end{document}